\documentclass[11pt]{article}
\usepackage{amsmath,amssymb,amsthm,bm,color}

\usepackage[colorlinks=true,anchorcolor=blue,filecolor=blue,linkcolor=red,urlcolor=blue,citecolor=blue]{hyperref}  
\usepackage[usenames,dvipsnames,svgnames,table]{xcolor}

\usepackage{graphicx,tikz,caption,subcaption}
\usetikzlibrary{patterns}
\usetikzlibrary{shapes}
\usepackage{fullpage}   
\usepackage{enumitem,verbatim}
\usepackage{multicol}
\usepackage{float}
\usepackage{booktabs}  

\usepackage[capitalise]{cleveref}

\usepackage{geometry}
\newtheorem{theorem}{Theorem}[section]
\newtheorem{corollary}[theorem]{Corollary}
\newtheorem{lemma}[theorem]{Lemma}
\newtheorem{proposition}[theorem]{Proposition}

\newtheorem{example}[theorem]{Example}
\newtheorem{claim}{Claim}

\theoremstyle{definition}
\newtheorem{definition}[theorem]{Definition}

\newtheorem{remark}{Remark}
\theoremstyle{plain}

\newcommand{\Ktt}{K_{t,t}}
\newcommand{\inj}{\mathrm{inj}}
\newcommand{\homc}{\mathrm{hom}}
\newcommand{\Aut}{\mathrm{Aut}}

\usepackage{todonotes}

\newenvironment{poc}{\begin{proof}[Proof of claim]}{\end{proof}}

\begin{document}

\title{Spectral Sidorenko inequalities and edge-spectral supersaturation}

\author{
Yongtao Li\thanks{Yau Mathematical Sciences Center, Tsinghua University, Beijing, China. Email: \url{ytli0921@hnu.edu.cn}.}
\and
Wilson Lin\thanks{Anysphere Inc., CA, USA. Email: \url{wilson@anysphere.co}.}
\and
Hong Liu\thanks{Extremal Combinatorics and Probability Group (ECOPRO), Institute for Basic Science (IBS), Daejeon, South Korea. Email: \url{hongliu@ibs.re.kr}. Supported by the Institute for Basic Science (IBS-R029-C4).}
\and
Shengtong Zhang\thanks{Anysphere Inc., CA, USA. Email: \url{shengtong@anysphere.co}.}
}

\date{\today}
\maketitle

\vspace{-0.5cm} 
\begin{abstract}
We develop a spectral approach to Sidorenko-type inequalities and apply
it to establish sharp edge-spectral supersaturation results.  Let \(H\) be a bipartite
graph with \(v\) vertices and \(e\) edges, where \(v\le e\), and write
\(M(G)=2e(G)\).  We prove that Sidorenko's conjecture is equivalent to a spectral
strengthening:
\[
        \hom(H,G)\ge M(G)^e |V(G)|^{v-2e} \quad \text{ if and only if }\quad \hom(H,G)\ge \lambda(G)^{2e-v}M(G)^{v-e}.
\]
We also introduce an operator-norm certificate which, via the 
Riesz--Thorin interpolation, gives direct proofs of the spectral
Sidorenko inequality in several cases.  The converse direction in the equivalence
theorem is proved by a tensor-power spectral regularization lemma.

{Our main result provides a unified framework to prove sharp asymptotic edge-spectral supersaturation results for degenerate bipartite graphs with the Sidorenko property, including complete bipartite graphs and even cycles.} 
Let $S_{t-1,m}$ be the split graph with $m$ edges 
obtained by joining a clique $K_{t-1}$ to an independent set. For every $m$-edge graph \(G\) with $\lambda(G)>\lambda(S_{t-1,m})$, 
$$\texttt{\#} K_{t,t}(G)
        \ge
        \Big(\frac{2^{-(t-1)^2}}{(t!)^2}-o(1)\Big)m^t \quad \text{and}\quad
 \texttt{\#}C_{2t}(G)
        \ge
        \Big(\frac{(t-1)!}{2t^t}-o(1)\Big)m^t.$$
Both constants are best possible: the first is attained asymptotically
by random graphs, while the second is
attained by split graphs.  The supersaturation proofs combine
spectral Sidorenko inequalities with a heavy-edge pruning process, a Perron-vector localized/delocalized dichotomy, and incidence-matrix inequalities.
\end{abstract}


\section{Introduction}

A central theme in extremal graph theory is to understand how global
density conditions force copies of a fixed graph.  For bipartite graphs,
one of the most influential principles of this kind is Sidorenko's
conjecture, which predicts that random graphs minimize homomorphism
counts among graphs with a fixed edge density.

This paper develops a spectral form of Sidorenko's inequality and applies
it to sharp edge-spectral supersaturation. The spectral
radius is the natural parameter in spectral extremal graph theory, and
in several basic cases the spectral form of Sidorenko's inequality is
more natural than the usual average-degree form.  For instance, for every $t\ge 2$, 
\begin{equation}\label{eq:even-cycle}
    \homc(C_{2t},G)
        =
        \operatorname{tr} A(G)^{2t}
        \ge
        \lambda(G)^{2t},
\end{equation}
which immediately gives the spectral Sidorenko inequality for even cycles.
For complete bipartite graphs, we prove the corresponding spectral Sidorenko 
inequality through an operator-norm certificate and Riesz--Thorin
interpolation.

Our first main theorem shows that, for bipartite graphs with at most as
many vertices as edges, the spectral form is actually equivalent to the
usual Sidorenko inequality.  Thus the spectral viewpoint is not a
strictly stronger conjecture, but rather a different and often useful
form of the same phenomenon.

We then use this equivalence as a counting tool in edge-spectral
supersaturation.
We prove sharp asymptotic lower bounds
for the number of copies of \(K_{t,t}\) and \(C_{2t}\) in any \(m\)-edge
graph whose spectral radius exceeds the relevant split-graph extremal
threshold.  A notable feature is that the two sharp constructions are
different: balanced random graphs are extremal for \(K_{t,t}\), whereas
split graphs are extremal for \(C_{2t}\).

\subsection{Spectral Sidorenko inequalities}

For two graphs \(G\) and \(H\), let \(\homc(H,G)\) be the number of graph
homomorphisms from \(H\) to \(G\), and let \(\inj(H,G)\) be the number of
labelled injective homomorphisms.  We write  $\texttt{\#}H(G):=\frac{\inj(H,G)}{|\Aut(H)|}$
for the number of unlabelled copies of \(H\) in \(G\).  We also write $|G|:=|V(G)|$ and $M(G):=2e(G).$

Sidorenko's conjecture~\cite{Sidorenko1993CorrelationInequality}
asserts that, for every bipartite graph \(H\) with \(v\) vertices and
\(e\) edges,
$        \homc(H,G)\ge |G|^{v-2e} M(G)^e $
for every graph \(G\).  Equivalently, if
        $p:=\frac{M(G)}{|G|^2}$ is the edge density of \(G\), then a uniformly random map
\(V(H)\to V(G)\) is a homomorphism with probability at least \(p^e\).

Let \(A(G)\) be the adjacency matrix of \(G\), and let \(\lambda(G)\) be
its spectral radius.  Since
        $\lambda(G)$ is at least the average degree  $\frac{M(G)}{|G|},$
it is natural to ask whether Sidorenko's inequality can be strengthened
by replacing part of the average-degree term with  the spectral radius.  For
a bipartite graph \(H\) with \(v\) vertices and \(e\) edges, the natural
spectral form is
       $ \homc(H,G)\ge \lambda(G)^{2e-v}M(G)^{v-e}.$
This inequality immediately implies Sidorenko's
inequality when $H$ has no isolated vertices, since $\lambda(G)\ge \frac{M(G)}{|G|}$.

Our first theorem shows that, perhaps surprisingly, the converse also holds.

\begin{theorem}[Equivalence of Sidorenko and spectral Sidorenko]
\label{thm:equivalence}
Let \(H\) be a bipartite graph with \(v\) vertices and \(e\) edges, where
\(v\le e\).  Then the following are equivalent.
\begin{itemize}
    \item[\rm(i)]
    \(H\) satisfies Sidorenko's inequality: for every graph \(G\), 
    $$\homc(H,G)\ge |G|^{v-2e}M(G)^e.$$

    \item[\rm(ii)]
    \(H\) satisfies the edge-spectral Sidorenko inequality: for every graph \(G\), 
    \begin{equation}
\label{eq:spectral-sidorenko-general}
        \homc(H,G)\ge \lambda(G)^{2e-v}M(G)^{v-e}.
\end{equation}

\item[\rm (iii)]
$H$ satisfies the vertex-spectral Sidorenko inequality: for every graph $G$, 
\[  \homc (H,G)\ge \lambda(G)^e |G|^{v-e}. \] 
\end{itemize}
\end{theorem}

We remark that the assumption \(v\le e\) in Theorem~\ref{thm:equivalence} is necessary.
The spectral Sidorenko inequality~\eqref{eq:spectral-sidorenko-general} fails already for \(H=P_3\). Indeed, considering $G:=K_{1,t}\sqcup t^2 K_2$, we have $\lambda (G)=\sqrt{t}$. Then $\lambda(G)M(G)=2t^{5/2}+O(t^{3/2})$ and 
$\lambda(G)^2|G|=2t^3+ O(t^2)$, whereas 
$\homc(P_3,G)=3t^2+O(t)$. Thus, for large \(t\), \(\homc(P_3,G)\ll \lambda(G)M(G) \ll \lambda(G)^2|G| \).

We also introduce a direct route to the spectral Sidorenko inequality~\eqref{eq:spectral-sidorenko-general}.  For \(v\le e\), set $s_H:=\frac{2e}{v}$ and $s_H':=\frac{2e}{2e-v}$.
We say that \(H\) satisfies the \emph{operator-norm Sidorenko
certificate} if
\[
        \homc(H,G)\ge \|A(G)\|_{s_H'\to s_H}^{\,e}
\]
for every graph \(G\).  The Riesz--Thorin interpolation shows that this
certificate implies the spectral Sidorenko inequality (Lemma~\ref{lem:operator-certificate}).  Hence, whenever the certificate can be verified directly, it gives an
analytic proof of the spectral Sidorenko inequality. For \(K_{t,t}\), applying H\"older's inequality gives the certificate 
$        \homc(K_{t,t},G)\ge \|A(G)\|_{t/(t-1)\to t}^{t^2}$ (Lemma~\ref{lem:operator-certificate-Ktt}). Interpolation then yields
\begin{equation}
\label{eq:Ktt}
        \homc(K_{t,t},G)
        \ge
        \frac{\lambda(G)^{2t(t-1)}}{M(G)^{t(t-2)}}.
\end{equation}
For $C_{2t}$, the operator-norm certificate is exactly~\eqref{eq:even-cycle}. Thus, the spectral form is not merely a formal consequence of Sidorenko's conjecture; in some cases it is the more transparent inequality to prove.

Theorem~\ref{thm:equivalence} allows us to transfer known cases of Sidorenko's conjecture~\cite{Sidorenko1993CorrelationInequality, ConlonFoxSudakov2010, Hatami2010, KimLeeLee2016, ConlonLee2017, ConlonKimLeeLee2018} directly to the spectral setting.  The operator-norm certificate also gives direct proofs for
complete bipartite graphs and even cycles.  For many weakly norming examples, including hypercubes and the
reflection-graph examples of Conlon--Lee~\cite{ConlonLee2017}, the
certificate can be obtained from the corresponding weighted Sidorenko inequalities or
the graph-norm H\"older inequalities.  Since weakly norming graphs are
already known to be Sidorenko, this should be viewed as a spectral
transfer principle rather than a new proof of Sidorenko for those
families.

The certificate is nevertheless not merely a reformulation of weakly
norming.  In Section~\ref{sec:proof-spectral-Sido}, we prove that it is
stable under \(1\)-subdivision (Proposition~\ref{prop:subdivision-closure-certificate}).  Consequently, since \(K_{a,b}\) satisfies
the certificate, so does its \(1\)-subdivision.  For \(a\ne b\), this
subdivision is not edge-transitive and hence cannot be weakly norming. Thus, the operator-norm certificate gives spectral
Sidorenko inequalities for examples outside the weakly norming class.

The main ingredient in the proof of Theorem~\ref{thm:equivalence} is a
spectral regularization lemma.

\begin{theorem}
\label{thm:regular-subgraphs}
Let \(G\) be a finite graph with \(M(G)>0\), and let
\(\lambda=\lambda(G)\).  Then there exists a constant \(C=C(G)>0\) such
that, for infinitely many positive integers \(k\), the tensor power
\(G^{\otimes k}\) contains a regular subgraph \(F_k\) of degree \(d_k\)
satisfying      $\lambda^k k^{-C}\le d_k\le \lambda^k .$
Moreover, $ M(F_k)\le M(G)^k.$
\end{theorem}

It is known that Sidorenko's conjecture admits reductions to
regular host graphs in the ordinary average-density setting; see
Coregliano--Razborov~\cite{CoreglianoRazborov2021Biregularity}, building
on Szegedy~\cite{Szegedy2015SparseGraphLimits}.  These reductions can not
directly apply to the spectral formulation, where the spectral radius may
be much larger than the average degree.  Theorem~\ref{thm:regular-subgraphs}
provides the needed spectral regularization: inside tensor powers, it
finds regular subgraphs whose degrees capture \(\lambda(G)\), rather
than the average degree \(M(G)/|G|\).

\subsection{Edge-spectral supersaturation}

We now turn to the application of spectral Sidorenko inequalities to
edge-spectral supersaturation.  Given a fixed graph \(F\), one asks how
many copies of \(F\) must appear in an \(m\)-edge graph \(G\) once
\(\lambda(G)\) exceeds the maximum possible spectral radius of an
\(F\)-free \(m\)-edge graph.

This is closely related to the Brualdi--Hoffman--Tur\'an problem, which
asks for the maximum spectral radius of a graph with a forbidden subgraph
and a given number of edges. 
For example, a well-known theorem of Nikiforov \cite{Niki2002} shows that every $K_{r+1}$-free graph $G$ with $m$ edges satisfies $\lambda(G)^2 \le (1-\frac{1}{r})2m$; see \cite{ChaoYu2026JLMS} for an entropic extension.   
There is now a substantial literature on
this edge-spectral Tur\'{a}n-type problem; see, e.g.,~\cite{ZLS2021,LZS2024,LLLY2025,LLZ2025,LiLiuZhang2025NosalBooksC4,ZhaiLiLou2026,LZZ2024}. 
Li--Liu--Zhang~\cite{LiLiuZhang2025EdgeSpectralESS}
proved an Erd\H{o}s--Stone--Simonovits type theorem in this setting:
if \(\chi(F)=r+1\ge3\), then every \(F\)-free graph \(G\) satisfies $\lambda(G)^2\le \left(1-\frac1r+o(1)\right)2m$. 
Moreover, Li--Liu--Zhang \cite{LiLiuZhang2025ColorCritical} extended Nikiforov's theorem to color-critical graphs, and they established some Tur\'{a}n-type results for almost-bipartite graphs, in which the extremal graphs are often split graphs. For integers \(k\ge1\) and \(m\ge\binom{k}{2}\), 
we write
$m-\binom{k}{2}=kq+r$ with $q\ge 0$ and $0\le r\le k-1.$ 
The {\it split graph} \(S_{k,m}\) is obtained by joining a clique \(K_k\) to an
independent set of size \(q\), and adding one further vertex adjacent to
exactly \(r\) vertices of the clique.  For complete bipartite graphs,
Li--Liu--Zhang~\cite{LiLiuZhang2025ColorCritical} proved the following
exact result.

\begin{theorem}[Li--Liu--Zhang~\cite{LiLiuZhang2025ColorCritical}]
\label{thm-Ktt-free}
For each fixed \(t\ge2\) and all sufficiently large \(m\), every
\(K_{t,t}\)-free \(m\)-edge graph \(G\) satisfies
       $ \lambda(G)\le \lambda(S_{t-1,m}),$
with equality if and only if \(G=S_{t-1,m}\).
\end{theorem}

In many classical extremal problems for bipartite forbidden graphs,
sharp or near-sharp constructions are pseudorandom or algebraic.  The
edge-spectral setting exhibits a different phenomenon: the extremal
configuration for forbidding \(K_{t,t}\) is the highly inhomogeneous
split graph \(S_{t-1,m}\).

Supersaturation asks what happens beyond such an extremal threshold.  The
first supersaturation theorem is due to Erd\H{o}s and
Rademacher~\cite{Erd1962a,Erdos1964}, who showed that every
\(n\)-vertex graph with more than \(\lfloor n^2/4\rfloor\) edges contains
at least \(\lfloor n/2\rfloor\) triangles.  Supersaturation has since
been extensively studied; see, for example,
\cite{LiuPikhurkoStaden2020,Reiher2016,Mubayi2010CountingSubstructuresI,
PikhurkoYilma2017Supersaturation,MY2025}.  Spectral supersaturation was
initiated by Bollob\'as--Nikiforov~\cite{BN2007jctb}; they proved that
every graph \(G\) contains at least
        $\frac13\lambda(G)\bigl(\lambda(G)^2-m\bigr)$
triangles. Recently, 
Fang--Lin--Zhai \cite{FangLinZhai2026} studied 
the edge-spectral supersaturation for color-critical subgraphs.  
Ning--Zhai~\cite{NZ2021,NZ2021b} proved that if
\(\lambda(G)>\sqrt m\), then \(G\) contains at least
\(\lfloor \frac{1}{2}(\sqrt m-1)\rfloor\) triangles and at least
\(\frac1{2000}m^2\) copies of \(C_4\).  The \(C_4\) bound was sharpened
by Li--Liu--Zhang~\cite{LiLiuZhang2025NosalBooksC4}.

\begin{theorem}[Li--Liu--Zhang~\cite{LiLiuZhang2025NosalBooksC4}]
\label{thm:counting-C4}
Every \(m\)-edge graph \(G\) with \(\lambda(G)>\sqrt m\) contains at
least        $\left(\frac18-o(1)\right)m^2$
copies of \(C_4\).  Moreover, the constant \(1/8\) is best possible.
\end{theorem}

The main contribution of this paper 
is to develop a unified method for proving
exact edge-spectral supersaturation results for degenerate bipartite graphs with the Sidorenko property.   
Our next result gives the sharp supersaturation theorem for \(K_{t,t}\) for every $t\ge 2$  
above the extremal spectral radius in Theorem~\ref{thm-Ktt-free}.  When
\(t=2\), this recovers Theorem~\ref{thm:counting-C4}, since
\(K_{2,2}=C_4\) and \(\lambda(S_{1,m})=\sqrt m\).

\begin{theorem}\label{thm:sharp-Ktt}
Let $t\ge 2$ and $G$ be an $m$-edge graph. If $\lambda(G)> \lambda(S_{t-1,m})$, then
\[
   \texttt{\#}\Ktt(G) \ge
   \left( \frac{2^{-(t-1)^2}}{(t!)^2} -o(1)\right)m^t .
\]
The constant is best possible, as witnessed by $m$-edge random graphs on
$(2+o(1))\sqrt m$ vertices.
\end{theorem}

For counting \(C_4\), the proof of Theorem~\ref{thm:counting-C4} relies crucially on an
exact spectral identity \cite{LL2009} for the number of \(4\)-cycles:  $\texttt{\#}C_4(G)
        =
        \frac18\sum_i(\lambda_i^4+\lambda_i^2)
        -
        \frac14\sum_i d_i^2,$
where \(\lambda_i\) are the eigenvalues of \(G\) and \(d_i\) are its
degrees. However, for \(K_{t,t}\) with \(t\ge3\), no analogous identity is available for controlling injective copies directly from the spectrum or degrees of $G$.  Instead, we need to use the spectral Sidorenko inequality
\eqref{eq:Ktt}, which gives the correct homomorphism count.  The new difficulty is to control degenerate homomorphisms to convert it into a lower bound for injective copies, especially
near extremal split  graphs where the Perron eigenvector can be highly localized.

The same structural framework also applies beyond complete bipartite graphs. As the second counting result, we prove the following sharp asymptotic supersaturation for even cycles $C_{2t}$ for every $t\ge 2$. Here the sharp construction is different from that of $K_{t,t}$. 

\begin{theorem}\label{thm:sharp-C2t}
 Let $t\ge 2$ and $G$ be an $m$-edge graph. If $\lambda(G)> \lambda(S_{t-1,m})$, then
 \[
   \texttt{\#}\, C_{2t}(G) \ge
   \left( \frac{(t-1)!}{2t^t} - o(1)\right) m^t .
 \]
 The constant is best possible, as witnessed by the split graphs $S_{t,m}$.
 \end{theorem}

Notice that the same spectral threshold leads to two different sharp extremal
mechanisms: balanced random graphs are sharp for \(K_{t,t}\), while split
graphs are sharp for \(C_{2t}\).

\paragraph{Proof overview.}

The proof of Theorem~\ref{thm:equivalence} starts from the observation
that ordinary Sidorenko and spectral Sidorenko coincide for regular host
graphs. The key is Theorem~\ref{thm:regular-subgraphs}, which provides a
spectral reduction to regular host graphs.  The construction for finding regular subgraph uses the Perron eigenvector to build a probability
distribution on ordered edges, i.e., 
        $p_{ij}:=\frac{a_{ij}x_ix_j}{\lambda(G)}.$
This distribution has both marginals equal to \(\pi_i=x_i^2\), which is a distribution on vertices. In the $k$-th tensor power \(G^{\otimes k}\), we take vertices with empirical distribution close to \(\pi\), and edges whose coordinate-wise empirical edge distribution is close to \(p\). The
matching marginal condition makes the resulting subgraph regular and the entropy identity \(H(p)-H(\pi)=\log\lambda(G)\) controls the degree. Applying Sidorenko's inequality to this regular subgraph, combining with multiplicativity and monotonicity, proves the spectral form.

We next outline the proof of Theorem~\ref{thm:sharp-Ktt}.  The spectral
Sidorenko inequality gives the correct homomorphism lower bound for
\(K_{t,t}\).  After subtracting non-injective homomorphisms, it gives the
sharp copy count whenever the relevant graph has \(O(\sqrt m)\) vertices; see Lemma \ref{lem:spectral-count-Ktt}. 
The main difficulty is that graphs above the split spectral threshold may
have localized Perron eigenvectors and may be close to split
\(K_{t,t}\)-free constructions.

We first prune the graph by deleting edges with small Perron weight 
product.  Such edges contribute little to the Rayleigh quotient, and a
discrete increment estimate for \(\lambda(S_{t-1,m})\) shows that the
strict gap above the split threshold is preserved in the resulting subgraph \(H\) and every remaining edge is heavy with respect to the Perron vector; see Lemma \ref{lem:lemmaA-pruning}. The pruned graph is then analyzed according to the parameter 
\[
        g:=\|\bm{x}\|_\infty e(H)^{1/4}.
\]
If \(g=O(1)\), then the heavy-edge condition forces the non-isolated part of
\(H\) to have \(O(\sqrt{e(H)})\) vertices. So the non-injective error $O(n^{2t-1})$ reduces to  $o(m^t)$, which is
negligible. Using spectral Sidorenko inequality  gives the sharp count; see Lemma \ref{lem:lemmaB-delocalized}.

The localized case \(g\to\infty\) requires a new structural theorem; see Theorem \ref{thm:localized-structure}. More precisely, we use
Perron-vector level sets to obtain a partition
        $V(H)=A\sqcup C\sqcup D,$
where \(D\) is independent, there are no \(C\)--\(D\) edges, every edge
is either inside \(C\) or incident to \(A\), and
$e_H(A,C)=o(e(H))$, $|A|=o(\sqrt{e(H)})$ and 
$ |C|\le e(H)^{1/2+o(1)}$. 
If the core \(H[C]\) contains a positive fraction of the edges, then a spectral cut estimate shows that
the core retains enough spectral radius, and spectral Sidorenko applies
inside \(H[C]\); see Lemma \ref{lem:lemmaC-localized-nonsparse}.  
If the core $H[C]$ is sparse, i.e., $e(H[C])=o(e(H))$, 
then almost all edges of $H$ lie between
\(A\) and \(D\).  The \(A\)--\(D\) incidence matrix is nearly rank one,
forcing either many aligned rows with a large common neighborhood, which
already produce many copies of \(K_{t,t}\), or else almost all edges are
covered by at most \(t-1\) vertices of \(A\); see Theorem \ref{thm:sparseC-rowcover}. 
The latter alternative is
close to a \(K_{t,t}\)-free split graph.  We rule it out by deleting the
exceptional edges, applying the edge-spectral Tur\'an theorem for
\(K_{t,t}\), and using the split-graph increment estimate to contradict
the preserved spectral gap; see the proof of Lemma \ref{lem:lemmaD-localized-sparse}. 

The proof of Theorem~\ref{thm:sharp-C2t} uses the same structural
framework, but the counting input is the simpler inequality~\eqref{eq:even-cycle}.
This leads to the split-graph constant \((t-1)!/(2t^t)\), whereas
\(K_{t,t}\) is governed by the balanced random-graph construction. In conclusion, the above discussions of $K_{t,t}$ and $C_{2t}$ illustrate that 
our method provides a unified framework for  establishing edge-spectral supersaturation
results for other bipartite graphs $H$ with the spectral Sidorenko property, including the bipartite graphs having a vertex complete to the other part, hypercubes and the $1$-subdivisions of cliques, upon a comparison of the number of  copies of $H$ in random graphs and split graphs. 

\paragraph{Organization.}

The rest of the paper is organized as follows.  In
Section~\ref{sec:proof-spectral-Sido}, we prove
Theorem~\ref{thm:equivalence} and develop the operator-norm certificate.
In Section~\ref{sec:preliminaries}, we collect preliminary estimates for
the \(K_{t,t}\) supersaturation theorem.  In
Section~\ref{sec-four-key-lemmas}, we prove Theorem~\ref{thm:sharp-Ktt}.
In Section~\ref{sec:even-cycles}, we prove Theorem~\ref{thm:sharp-C2t}.
Finally, in Section~\ref{sec:Concluding}, we discuss several concluding
questions.

\section{Spectral Sidorenko inequality}\label{sec:proof-spectral-Sido}

\subsection{Spectral Sidorenko via an operator-norm certificate}

In this subsection we introduce an analytic sufficient condition for
proving spectral Sidorenko inequalities.  The condition is formulated in
terms of operator norms and Riesz--Thorin interpolation.
We verify it for \(K_{t,t}\), which is the main counting input in the proof of Theorem~\ref{thm:sharp-Ktt}.

Let \(\bm{x}\in\mathbb R^n\).  For \(1\le p<\infty\), write
        $\|\bm{x}\|_p
        :=
        \left(\sum_{i=1}^n |x_i|^p\right)^{1/p}.$
We also write
        $\|\bm{x}\|_\infty:=\max_{1\le i\le n}|x_i|.$ For a real
\(n\)-by-\(n\) matrix \(A\), define the \((p,r)\)-operator norm by $ \|A\|_{p\to r}
        :=
        \sup_{\bm{x}\neq 0}
        \frac{\|A\bm{x}\|_r}{\|\bm{x}\|_p},
$, $1\le p,r\le\infty$.
We shall use the following standard form of the Riesz--Thorin
interpolation theorem; see, for example,
Cerda~\cite[page~61]{Cerda2010linear} and
Grafakos~\cite[page~37]{Grafakos2014ClassicalFourierAnalysis}.

\begin{theorem}[Riesz--Thorin interpolation]
\label{thm:riesz-thorin}
Let \(T\) be a linear operator between finite-dimensional \(\ell^p\)
spaces.  Suppose that
       $ \|T\|_{p_0\to q_0}\le M_0$ and 
       $ \|T\|_{p_1\to q_1}\le M_1,$
where \(p_0,q_0,p_1,q_1\in[1,\infty]\), with the convention
\(1/\infty=0\).  If \(\theta\in[0,1]\) and
        $\frac1{p_\theta}
        =
        \frac{1-\theta}{p_0}
        +
        \frac{\theta}{p_1},$
$        \frac1{q_\theta}
        =
        \frac{1-\theta}{q_0}
        +
        \frac{\theta}{q_1},
$
then
\[
        \|T\|_{p_\theta\to q_\theta}
        \le
        M_0^{1-\theta}M_1^\theta .
\]
\end{theorem}

We now define the operator-norm certificate.

\begin{definition}[Operator-norm certificate]
Let \(H\) be a bipartite graph with \(v\) vertices and \(e\) edges, where
\(v\le e\).  Set
        $s_H:=\frac{2e}{v}$ and 
$        s_H':=\frac{2e}{2e-v}.
$
Thus \(1/s_H+1/s_H'=1\).  We say that \(H\) satisfies the
\emph{operator-norm Sidorenko certificate} if, for every graph \(G\) with
adjacency matrix \(A=A(G)\),
\[
        \homc(H,G)\ge \|A\|_{s_H'\to s_H}^{\,e}.
\]
\end{definition}

The certificate is designed so that interpolation converts it into the
spectral Sidorenko inequality.

\begin{lemma}[Operator-norm certificate implies spectral Sidorenko]
\label{lem:operator-certificate}
Let \(H\) be a bipartite graph with \(v\) vertices and \(e\) edges, where
\(v\le e\).  If \(H\) satisfies the operator-norm Sidorenko certificate,
then, for every graph \(G\) with \(M(G)=2e(G)>0\),
$        \homc(H,G)
        \ge
        \lambda(G)^{2e-v}M(G)^{v-e}.$
\end{lemma}

\begin{proof}
Let \(A=A(G)\), \(M=M(G)\), and
\(\lambda=\lambda(G)=\|A\|_{2\to2}\).  Put
        $\alpha:=\frac{e}{2e-v}.$
Since \(v\le e\), we have \(0<\alpha\le1\).  Moreover,
        $\frac12
        =
        \alpha\frac1{s_H'}
        +
        (1-\alpha)\frac1\infty$ and
        $\frac12
        =
        \alpha\frac1{s_H}
        +
        (1-\alpha)\cdot 1.$
Applying Riesz--Thorin interpolation between
\(\|A\|_{s_H'\to s_H}\) and \(\|A\|_{\infty\to1}\), and using
       $ \|A\|_{\infty\to1}\le M,$
we obtain
\[
        \lambda
        =
        \|A\|_{2\to2}
        \le
        \|A\|_{s_H'\to s_H}^{\alpha}M^{1-\alpha}.
\]
Hence
     $   \|A\|_{s_H'\to s_H}
        \ge
        \lambda^{1/\alpha}M^{-(1-\alpha)/\alpha}.$
Raising both sides to the \(e\)-th power gives
\[
        \|A\|_{s_H'\to s_H}^{e}
        \ge
        \lambda^{e/\alpha}M^{-e(1-\alpha)/\alpha}
        =
        \lambda^{2e-v}M^{v-e}.
\]
The operator-norm certificate now gives the desired inequality.
\end{proof}

More generally, if \(W=(w_{ij})\) is a nonnegative symmetric matrix indexed by a finite set, we write
\[
        \homc(H,W)
        :=
        \sum_{\phi:V(H)\to [n]}
        \prod_{uv\in E(H)} w_{\phi(u)\phi(v)}
\]
for the corresponding weighted homomorphism count.

We next verify the certificate for \(K_{t,t}\), in which case 
        $s_{K_{t,t}}=t$ and $
        s_{K_{t,t}}'=\frac{t}{t-1}.$
In fact, the proof below works verbatim for every nonnegative
symmetric matrix \(A\).

\begin{lemma}[Operator-norm certificate for \(K_{t,t}\)]
\label{lem:operator-certificate-Ktt}
For every \(t\ge2\) and every graph \(G\) with adjacency matrix \(A\),
\[
        \homc(K_{t,t},G)\ge \|A\|_{q\to t}^{t^2},
        \qquad
        q:=\frac{t}{t-1}.
\]
Consequently, if \(M(G)>0\), then
$\homc(K_{t,t},G)
        \ge
        \frac{\lambda(G)^{2t(t-1)}}{M(G)^{t(t-2)}}.
$
\end{lemma}

\begin{proof}
Let \(A=(a_{ij})\) be the adjacency matrix of \(G\).  Since \(A\) has
nonnegative entries, the norm \(\|A\|_{q\to t}\) may be computed using
nonnegative vectors, by replacing \(\bm{x}\) with \(|\bm{x}|\).
Let \(\bm{x}\ge0\).  For
\(J=(j_1,\ldots,j_t)\in[n]^t\), write
\[
        S_J
        :=
        \sum_{i=1}^n\prod_{b=1}^t a_{ij_b}.
\]
Then
$        \homc(K_{t,t},G)
        =
        \sum_{J\in[n]^t}S_J^t.$
Expanding \(\|A\bm{x}\|_t^t\) and applying H\"older's inequality on
\([n]^t\), we obtain
\[
\begin{aligned}
        \|A\bm{x}\|_t^t
        &=
        \sum_{i=1}^n
        \left(\sum_{j=1}^n a_{ij}x_j\right)^t                         \\
        &=
        \sum_{j_1,\ldots,j_t=1}^n
        \left(\prod_{b=1}^t x_{j_b}\right)
        \left(\sum_{i=1}^n\prod_{b=1}^t a_{ij_b}\right)                \\
        &\le
        \left(\sum_{J\in[n]^t}S_J^t\right)^{1/t}
        \left(
        \sum_{j_1,\ldots,j_t=1}^n
        \prod_{b=1}^t x_{j_b}^q
        \right)^{1/q}                                                  \\
        &=
        \homc(K_{t,t},G)^{1/t}\|\bm{x}\|_q^t .
\end{aligned}
\]
Taking \(t\)-th roots gives
$ \|A\bm{x}\|_t
        \le
        \homc(K_{t,t},G)^{1/t^2}\|\bm{x}\|_q .
$
Taking the supremum over all nonzero \(\bm{x}\ge0\), and using again
that \(A\) is nonnegative, yields the desired inequality
$        \|A\|_{q\to t}
        \le
        \homc(K_{t,t},G)^{1/t^2}.
$
\end{proof}

\begin{remark}[Other examples and scope of the certificate]
Even cycles also satisfy the certificate directly.  For \(H=C_{2t}\),
we have \(v=e=2t\), so \(s_H=s_H'=2\).  The certificate becomes exactly~\eqref{eq:even-cycle}: 
$        \homc(C_{2t},G)
        =
        \operatorname{tr}A(G)^{2t}
        \ge
        \lambda(G)^{2t}
        =
        \|A(G)\|_{2\to2}^{2t}.
$ The operator-norm certificate is a sufficient condition, not an
equivalent reformulation of Sidorenko's conjecture.  Its usefulness is
that it produces the exact type of operator-norm lower bound needed in
spectral counting problems.
\end{remark}

\begin{proposition}[Subdivision closure]
\label{prop:subdivision-closure-certificate}
Let \(H\) be a bipartite graph with \(v\) vertices and \(e\) edges, where
\(v\le e\).  Set $       s:=\frac{2e}{v}$ and 
$        s':=\frac{2e}{2e-v}.
$
Assume that \(H\) satisfies the matrix version of the operator-norm
certificate: for every nonnegative symmetric matrix \(W\),
\[
        \homc(H,W)\ge \|W\|_{s'\to s}^{\,e}.
\]
Let \(\operatorname{sub}(H)\) be the graph obtained by subdividing every
edge of \(H\) once.  Then \(\operatorname{sub}(H)\) satisfies the
operator-norm certificate.  More precisely, for every graph \(G\) with
adjacency matrix \(A=A(G)\),
\[
        \homc(\operatorname{sub}(H),G)
        \ge
        \|A\|_{\sigma'\to\sigma}^{\,2e},
\]
where        $\sigma=\frac{4e}{v+e}$ and 
$        \sigma'=\frac{4e}{3e-v}.
$
\end{proposition}

\begin{proof}
Since each subdivided edge corresponds to a length-two walk in \(G\),
\[
        \homc(\operatorname{sub}(H),G)=\homc(H,A^2).
\]
By the matrix certificate for \(H\),
$ \homc(\operatorname{sub}(H),G)
        \ge
        \|A^2\|_{s'\to s}^{\,e}.$
Let \(N:=\|A^2\|_{s'\to s}\).  For every vector \(\bm{x}\),
\[
        \|A\bm{x}\|_2^2
        =
        \bm{x}^{\mathsf T}A^2\bm{x}
        \le
        \|A^2\bm{x}\|_s\|\bm{x}\|_{s'}
        \le
        N\|\bm{x}\|_{s'}^2.
\]
Thus   $ \|A\|_{s'\to2}\le N^{1/2}.$
By duality and symmetry of \(A\),
$        \|A\|_{2\to s}\le N^{1/2}.
$
Interpolating these two estimates with parameter \(1/2\), we get $\|A\|_{\sigma'\to\sigma}\le N^{1/2},$
since
$\frac1{\sigma'}
        =
        \frac12\cdot\frac1{s'}
        +
        \frac12\cdot\frac12
$ and $\frac1{\sigma}
        =
        \frac12\cdot\frac12
        +
        \frac12\cdot\frac1s.$
Therefore
\[
        \|A^2\|_{s'\to s}\ge \|A\|_{\sigma'\to\sigma}^2.
\]
Combining the inequalities gives
\[
        \homc(\operatorname{sub}(H),G)
        \ge
        \|A\|_{\sigma'\to\sigma}^{2e}.
\]
Since \(\operatorname{sub}(H)\) has \(v+e\) vertices and \(2e\) edges, its certificate exponents are precisely
        $\frac{2\cdot 2e}{v+e}=\sigma$ and 
$        \frac{2\cdot 2e}{2\cdot 2e-(v+e)}=\sigma'.
$
Thus this is exactly the required certificate.
\end{proof}

\begin{corollary}
The \(1\)-subdivision of \(K_{a,b}\) satisfies the operator-norm
Sidorenko certificate for all \(a,b\ge2\).  In particular, for every \(t\ge2\),
\[
        \homc(\operatorname{sub}(K_{t,t}),G)
        \ge
        \|A(G)\|_{\frac{4t}{3t-2}\to \frac{4t}{t+2}}^{2t^2}.
\]
\end{corollary}

\begin{proof}
The matrix version of the certificate for \(K_{a,b}\) follows from the
same H\"older argument used above for \(K_{t,t}\).  Applying
Proposition~\ref{prop:subdivision-closure-certificate} gives the first
claim.  For \(a=b=t\), the \(1\)-subdivision has \(2t+t^2\) vertices and
\(2t^2\) edges, so its certificate exponents are
$\sigma=\frac{4t^2}{2t+t^2}=\frac{4t}{t+2}$ and $       \sigma'=\frac{4t^2}{4t^2-(2t+t^2)}
        =
        \frac{4t}{3t-2}.$
\end{proof}

\begin{remark}[Comparison with domination]
Conlon--Lee proved that \(\operatorname{sub}(K_{t,t})\) is dominating.
Their proof uses a reflection/relocation mechanism, whereas
Proposition~\ref{prop:subdivision-closure-certificate} is a direct
operator-norm closure argument.  The conclusions are different:
domination gives a family of homomorphism inequalities, while the
operator-norm certificate is tailored to spectral Sidorenko inequalities.
The same subdivision-closure argument applies to any graph satisfying the
matrix version of the certificate.  In particular, it applies to weakly norming incidence graphs for which the matrix certificate is available. 
\end{remark}

\subsection{Equivalence of Sidorenko and spectral Sidorenko}

We now prove Theorem~\ref{thm:regular-subgraphs} and use it to derive Theorem \ref{thm:equivalence}.  The proof uses a tensor-power
argument, a useful tool in extremal combinatorics; see, for example,
\cite{ConlonFoxSudakov2010, FoxLovasz2017Adv, LovaszSauermann2019PLMS,Tao2008TensorPower}. The \emph{tensor product} \(G\otimes H\) of two graphs \(G\) and \(H\)
has vertex set \(V(G)\times V(H)\), where \((u_1,v_1)\) is adjacent to
\((u_2,v_2)\) if and only if \(u_1u_2\in E(G)\) and \(v_1v_2\in E(H)\).
The \(k\)-th tensor power of \(G\), denoted by \(G^{\otimes k}\), is
defined recursively by \(G^{\otimes 1}=G\) and
\(G^{\otimes k}=G^{\otimes(k-1)}\otimes G\).  Equivalently,
\(G^{\otimes k}\) has vertex set \(V(G)^k\), and two \(k\)-tuples
\(\bm u=(u_1,\dots,u_k)\) and \(\bm v=(v_1,\dots,v_k)\) are adjacent if
and only if \(u_i v_i\in E(G)\) for every \(i\in\{1,\dots,k\}\).

\begin{proof}[Proof of Theorem \ref{thm:regular-subgraphs}]
Let \(G_0\) be a connected component of \(G\) with vertex set \(V_0\)
and spectral radius \(\lambda(G_0)=\lambda (G)=\lambda \).  Let
\(A=(a_{ij})_{i,j\in V_0}\) be the adjacency matrix of \(G_0\), and let
\(\bm x=(x_i)_{i\in V_0}\) be a positive Perron eigenvector.  We scale
\(\bm x\) so that $\sum_{i\in V_0}x_i^2=1$. 

For ordered pairs \((i,j)\in V_0^2\), we define
\[
        p_{ij}:=\frac{a_{ij}x_i x_j}{\lambda}.
\]
Then \(p_{ij}=0\) if and only if \(ij\notin E(G_0)\).  Moreover, we have $ \sum_{i,j\in V_0}p_{ij}
        =
        \frac{1}{\lambda} \bm x^{\mathsf T}A\bm x
        =
        1$. 
        
For each \(i\in V_0\), we write 
\[ \pi_i:=x_i^2. \] 
Then $\sum_{i\in V_0} \pi_i =1$. Moreover, we have the following relation:  
\begin{equation} \label{eq:sum_j}
        \sum_j p_{ij}
        =
        \frac{x_i}{\lambda}\sum_j a_{ij}x_j
        =
        x_i^2 =\pi_i.
\end{equation}
In the above, we defined two probability distributions $p=(p_{ij})_{i,j\in V_0}$ and $\pi=(\pi_i)_{i\in V_0}$. 
We next show an interesting connection between them and the  spectral radius.
Let
$ H(\alpha):=-\sum_i \alpha_i\log\alpha_i$ 
be the entropy of a distribution $\alpha$, with the convention \(0\log0=0\). 
 We claim that 
\begin{equation} \label{eq-connection}
        H(p)-H(\pi)=\log\lambda .
\end{equation}
Indeed, by definition, we immediately get 
\[
\begin{aligned}
        H(p)-H(\pi)
        &=
        -\sum_{i,j}p_{ij}
        \log\left(\frac{p_{ij}}{\pi_i}\right)                              \\
        &=
        -\sum_{i,j:p_{ij}>0}p_{ij}
        \log\left(\frac{x_j}{\lambda x_i}\right)                            \\
        &=
        \log\lambda
        +
        \sum_{i,j}p_{ij}\log x_i
        -
        \sum_{i,j}p_{ij}\log x_j                                            \\
        &=
        \log\lambda
        +
        \sum_i \pi_i\log x_i
        -
        \sum_j \pi_j\log x_j                                                \\
        &=
        \log\lambda .
\end{aligned}
\]

We now construct a subgraph of \(G_0^{\otimes k}\).  
We do this along even values of \(k\).  Write \(k=2s\).  Since
\(p_{ij}=p_{ji}\), define a probability distribution on the unordered edges of \(G_0\) by
        $q_{ij}:=2p_{ij}$ for each $ij\in E(G_0).$
For each sufficiently large integer \(s\), choose nonnegative integers
\(m_{ij}^{(s)}\), one for each unordered edge \(ij\in E(G_0)\), such that $m_{ij}^{(s)}=s q_{ij}+O_G(1)$ and 
        $\sum_{ij\in E(G_0)}m_{ij}^{(s)}=s$. This can be done by rounding the numbers \(s q_{ij}\) and adjusting
finitely many entries.

Define a symmetric integer matrix
\(N^{(k)}=(N_{ij}^{(k)})_{i,j\in V_0}\) by
\[ N_{ij}^{(k)}=N_{ji}^{(k)}:=m_{ij}^{(k)}
\quad\text{if }ij\in E(G_0);
\qquad
  N_{ij}^{(k)}:=0
\quad\text{otherwise}.
\]
Then $\sum_{i,j}N_{ij}^{(k)}=k$. 
For all \(i,j\in V_0\), we have ${N_{ij}^{(k)}}/{k}\to p_{ij}$ as $k\to\infty$.

For each \(i\in V_0\), we denote 
\[
        n_i^{(k)}:=\sum_j N_{ij}^{(k)}.
\]
Then $\sum_{i\in V_0} n_i^{(k)}=k$, and for each $i\in V_0$, using (\ref{eq:sum_j}) yields   
${n_i^{(k)}}/ {k}\to \pi_i$ as $k\to \infty$.

Let \(T_k\) be the set of all \(k\)-tuples $\bm a=(a_1,\dots,a_k)\in V_0^k$ in which the vertex \(i\) occurs exactly \(n_i^{(k)}\) times.  Define a
graph \(F_k\) with vertex set \(T_k\) as follows.  Two vertices
\(\bm a=(a_1,\dots,a_k)\) and \(\bm b=(b_1,\dots,b_k)\) are adjacent in
\(F_k\) if, for every ordered pair \((i,j)\), the number of coordinates
\(r\in\{1,\dots,k\}\) such that
$  (a_r,b_r)=(i,j)$ 
is exactly \(N_{ij}^{(k)}\).

First, we see that \(F_k\subseteq G_0^{\otimes k}\).  Indeed, if
\(\bm a\bm b\in E(F_k)\), then every coordinate pair \((a_r,b_r)\) has
\(N_{a_r b_r}^{(k)}>0\), and hence \(a_r b_r\in E(G_0)\).  Therefore
\(\bm a\) and \(\bm b\) are adjacent in \(G_0^{\otimes k}\).

Second, \(F_k\) is regular.  Fix \(\bm a\in T_k\).  For each \(i\in V_0\),
there are \(n_i^{(k)}\) coordinates \(r\) with \(a_r=i\).  To choose a
neighbor \(\bm b\), we must partition these \(n_i^{(k)}\) coordinates
into classes of sizes \(N_{ij}^{(k)}\), one class for each \(j\in V_0\),
and assign \(b_r=j\) on the class corresponding to \(j\).  Thus the
number of choices is
\[
        d_k
        =
        \prod_i
        \frac{n_i^{(k)}!}{\prod_j N_{ij}^{(k)}!}.
\]
Since \(N^{(k)}\) is symmetric, the number of coordinates in which
\(\bm b\) is equal to \(j\) is
\[
        \sum_i N_{ij}^{(k)}
        =
        \sum_i N_{ji}^{(k)}
        =
        n_j^{(k)}.
\]
Hence \(\bm b\in T_k\), and the above number is independent of
\(\bm a\).  Therefore, \(F_k\) is \(d_k\)-regular.

Finally, we estimate \(d_k\).   Applying Stirling's
formula to the expression for \(d_k\), and using that \(G\) is fixed, we
obtain
\[
\begin{aligned}
        \log d_k
        &=
        \sum_i \log(n_i^{(k)}!)
        -
        \sum_{i,j}\log(N_{ij}^{(k)}!)                                      \\
        &\ge
        \sum_i n_i^{(k)}\log n_i^{(k)}
        -
        \sum_{i,j}N_{ij}^{(k)}\log N_{ij}^{(k)}
        -
        C\log k                                                            \\
        &\ge
        k\bigl(-H(\pi) + H(p) \bigr)-C\log k,
\end{aligned}
\]
where \(C=C(G)>0\) is a constant depending only on $G$.

Applying (\ref{eq-connection}) gives $\log d_k\ge k\log\lambda-C\log k$,
and therefore  $d_k\ge \lambda^k k^{-C}. $
The upper bound on \(d_k\) follows from monotonicity of the spectral
radius.  Since \(F_k\) is \(d_k\)-regular, we have $d_k=\lambda(F_k)\le \lambda(G_0^{\otimes k})=\lambda(G_0)^k=\lambda^k$.
Since \(F_k\subseteq G_0^{\otimes k}\subseteq G^{\otimes k}\), we get $M(F_k)\le M(G^{\otimes k})=M(G)^k$ immediately.
This completes the proof.
\end{proof}

We now prove the equivalence theorem.

\begin{proof}[Proof of Theorem~\ref{thm:equivalence}]
Assume throughout that \(v\le e\), and consider an arbitrary host graph \(G\) with
\(M(G)>0\). For notational convenience, we denote $M:=M(G)$ and $\lambda:=\lambda(G).$

First, we show that (ii) and (iii) implies (i).  Assume that \(H\) satisfies the edge-spectral Sidorenko inequality.
The Rayleigh quotient implies $\lambda\ge \frac{M}{|G|}.$
Since \(v\le e\), we have \(2e-v\ge0\).  Hence
\[
\homc(H,G)\ge\lambda^{2e-v}M^{v-e}\ge
        \left(\frac{M}{|G|}\right)^{2e-v}M^{v-e}=
        M^e |G|^{v-2e}.
\]
Thus \(H\) satisfies Sidorenko's inequality. Similarly, the vertex-spectral version implies 
\[ \homc (H,G) \ge \lambda^{e}|G|^{v-e} \ge 
\left( \frac{M}{|G|}\right)^e |G|^{v-e} = M^e |G|^{v-2e}. \]

Conversely, assume that \(H\) satisfies Sidorenko's inequality in (i). Next, we show that $H$ satisfies the spectral versions in (ii) and (iii).   
Let \(F_k\subseteq G^{\otimes k}\) be the regular subgraphs supplied by
Theorem~\ref{thm:regular-subgraphs}, along an infinite sequence of
integers \(k\).  Let $N_k:=|F_k|$, $M_k:=M(F_k)$,
and let \(d_k\) be the regular degree of \(F_k\).  Since \(F_k\) is
regular, we have $M_k=N_kd_k.$

The homomorphism count is multiplicative under tensor products, and
monotone under passing to subgraphs.  Therefore
\[
        \homc(H,G)^k
        =
        \homc(H,G^{\otimes k})
        \ge
        \homc(H,F_k).
\]
Applying Sidorenko's inequality to \(H\) with host graph \(F_k\), we get
\begin{equation} \label{eq-regular-sub}
\homc(H,F_k)
        \ge
        M_k^e N_k^{v-2e}                                      =
        M_k^e\left(\frac{M_k}{d_k}\right)^{v-2e}              =
        d_k^{2e-v}M_k^{v-e}.
\end{equation}
Since \(v\le e\), we have \(v-e\le0\).  Theorem~\ref{thm:regular-subgraphs}
gives \(M_k\le M^k\), and therefore $M_k^{v-e}\ge M^{k(v-e)}$.
It follows that
$        \homc(H,G)^k
        \ge
        d_k^{2e-v}M^{k(v-e)}.$
Using \(d_k\ge \lambda^k k^{-C}\), we obtain $\homc(H,G)^k
        \ge
        \left(\lambda^k k^{-C}\right)^{2e-v}M^{k(v-e)}$.
Taking \(k\)-th roots gives
\[
        \homc(H,G)
        \ge
        \lambda^{2e-v}k^{-C(2e-v)/k}M^{v-e}.
\]
Letting \(k\to\infty\), we get the edge-spectral Sidorenko inequality $\homc(H,G)\ge \lambda^{2e-v}M^{v-e}$. 

The vertex-spectral version is similarly obtained by substituting $M_k=N_kd_k$ into (\ref{eq-regular-sub}). Indeed, 
\[ \homc(H,F_k) \ge M_k^e N_k^{v-2e} = (N_kd_k)^e N_k^{v-2e} =N_k^{v-e} d_k^e. \]
Since $v-e\le 0$, $N_k\le |G|^k$ and $d_k \ge \lambda^k k^{-C}$ by Theorem~\ref{thm:regular-subgraphs}, we get $\homc (H,G)\ge |G|^{v-e} \lambda^{e} k^{-Ce/k}$. Taking $k\to \infty$, we  obtain the vertex-spectral version $\homc(H,G) \ge \lambda^e |G|^{v-e}$, as desired.    
\end{proof}

\section{Preliminary results for Theorem \ref{thm:sharp-Ktt}}

\label{sec:preliminaries}

For \(U\subseteq V(G)\), we write \(G[U]\) for the subgraph of \(G\)
induced by \(U\).  For disjoint sets \(X,Y\subseteq V(G)\), we write
\(E_G(X,Y)\) for the set of edges between \(X\) and \(Y\), and set
\(e_G(X,Y):=|E_G(X,Y)|\).  For two disjoint vertex sets \(A,B\), we
write \(G[A,B]\) for the bipartite subgraph of \(G\) with partite sets
\(A\) and \(B\).  For graphs \(G\) and \(H\) on disjoint vertex sets,
the \emph{join} \(G\vee H\) is obtained from \(G\sqcup H\) by adding
all edges between \(V(G)\) and \(V(H)\).  A \emph{split partition} of a
graph \(G\) means a partition \(V(G)=A\sqcup B\) such that \(B\) is
independent.

\subsection{Bounding \(\texttt{\#}K_{t,t}\) via spectral radius}
The spectral Sidorenko inequality \eqref{eq:Ktt} gives a lower bound on
\(\homc(K_{t,t},G)\).  By subtracting the non-injective homomorphisms,
we obtain the following lower bound for genuine copies of \(K_{t,t}\).

\begin{lemma}
\label{lem:spectral-count-Ktt}
Let \(t\ge 2\) and \(G\) be an \(n\)-vertex $m$-edge graph with spectral
radius \(\lambda\). Then        
\[
  \texttt{\#}\Ktt(G)
  \ge
  B_t\left(\frac{\lambda^2}{m}\right)^{t(t-1)}m^t
  -
  \frac{\binom{2t}{2}}{2(t!)^2}\,n^{2t-1}, \qquad\text{where}\quad B_t:=\frac{2^{-(t-1)^2}}{(t!)^2}. 
\]
\end{lemma}

\begin{proof}
By (\ref{eq:Ktt}), we get $\homc(\Ktt,G)
  \ge
  \frac{\lambda^{2t(t-1)}}{(2m)^{t(t-2)}}$. The number of non-injective maps \(V(\Ktt)\to V(G)\) is at most
\(\binom{2t}{2}n^{2t-1}\), since one may first choose a pair of vertices
of \(K_{t,t}\) that are identified.  Hence
\[
  \inj(\Ktt,G)
  \ge
  \homc(\Ktt,G)-\binom{2t}{2}n^{2t-1}.
\]
Since \(|\Aut(\Ktt)|=2(t!)^2\), we have
\[
  \texttt{\#}\Ktt(G)
  =
  \frac{\inj(\Ktt,G)}{2(t!)^2}
  \ge
  \frac{1}{2(t!)^2}
  \cdot
  \frac{\lambda^{2t(t-1)}}{(2m)^{t(t-2)}}
  -
  \frac{\binom{2t}{2}}{2(t!)^2}\,n^{2t-1}.
\]
The main term is
\[
  \frac{1}{2(t!)^2}
  \cdot
  \frac{\lambda^{2t(t-1)}}{(2m)^{t(t-2)}}
  =
  B_t\left(\frac{\lambda^2}{m}\right)^{t(t-1)}m^t,
\]
as desired.
\end{proof}

\subsection{Matrix inequalities for graphs}
\label{subsec:prelim-spectral}

We use standard spectral graph conventions; see, for example,
Chung~\cite{Chung1997SpectralGraphTheory} and
Horn--Johnson~\cite{HornJohnson2012MatrixAnalysis}.  For a graph \(G\),
let \(A(G)\) be its adjacency matrix and let \(\lambda(G)\) denote its
adjacency spectral radius.  For a real matrix \(A\), let \(\|A\|\) denote
the operator norm, $\|A\|
        =
        \max_{\bm{x}\neq 0}
        \frac{\|A\bm{x}\|_2}{\|\bm{x}\|_2}.$
We denote the Frobenius norm by $\|A\|_{\mathrm F}
        :=
    (\operatorname{tr}(A^{\mathsf T}A) )^{1/2}.$
Let $\sigma_1(A)\ge \cdots \ge \sigma_r(A)$
be the singular values of \(A\), where \(r:=\operatorname{rank}(A)\).
Then $\|A\|=\sigma_1(A),$ and $
        \|A\|_{\mathrm F}
=(\sum_{i=1}^r\sigma_i^2(A))^{1/2}.$
Since \(A(G)\) is symmetric and nonnegative, \(\lambda(G)=\|A(G)\|\).
The Perron--Frobenius theorem implies that \(\lambda(G)\) has an
eigenvector with nonnegative entries; if \(G\) is connected, then such an
eigenvector has strictly positive entries.  Throughout the paper, a
\emph{unit Perron eigenvector} of \(G\) means a vector
\(\bm{x}\in\mathbb R^{V(G)}\) satisfying
$\bm{x}\ge 0,$ $
        \|\bm{x}\|_2=1,$ and $
        A(G)\bm{x}=\lambda(G)\bm{x}.$ The Rayleigh quotient gives
\[
        \lambda(G)
        =
        \max_{\|\bm{y}\|_2=1}
        \bm{y}^{\mathsf T}A(G)\bm{y}.
\]
In particular, if \(H\subseteq G\), then \(\lambda(H)\le \lambda(G)\).
Moreover, if \(G\) has \(n\) vertices and \(m\) edges, then
\[
        \frac{2m}{n}\le \lambda(G)\le \|A(G)\|_{\mathrm F}=\sqrt{2m}.
\]

Let \(G\) be a graph with unit Perron eigenvector \(\bm{x}\).  For
\(U\subseteq V(G)\), we write
\[
        \mu(U):=\sum_{u\in U}x_u^2
\]
and call it the \emph{Perron mass} of \(U\).  We refer to
\cite{LiuNing2026, CLLLN2026} for recent results on Perron mass when
\(U\) is an independent set.

\begin{lemma}
\label{lem:spectral-vertex-cut}
Let \(H\) be a graph and let \(V(H)=U\sqcup W\) be a partition.
Let \(M\in\{0,1\}^{U\times W}\) be the \(U\)--\(W\) incidence matrix with $ m_{UW}:=e_H(U,W)=\|M\|_{\mathrm F}^2$ and  
$\rho:=\|M\|$. Set
\[ \lambda:=\lambda(H),
        \quad
        \lambda_U:=\lambda(H[U]),
        \quad
        \lambda_W:=\lambda(H[W]).
\]
Then the following hold.
\begin{enumerate}
  \item[\rm(a)] $\lambda
        \le
        \max\{\lambda_U,\lambda_W\}+\rho
        \le
        \max\{\lambda_U,\lambda_W\}+\sqrt{m_{UW}}.$

  \item[\rm(b)]
        $(\lambda-\lambda_U)(\lambda-\lambda_W)
        \le
        \rho^2
        \le
        m_{UW}.$

  \item[\rm(c)]
  Let \(\bm{x}\) be a unit Perron eigenvector of \(H\), and let
  \(\mu(U)=\sum_{u\in U}x_u^2\) be the Perron mass of \(U\) with respect
  to \(\bm{x}\).  If $(\lambda-\lambda_U)^2+\rho^2>0,$
  then
  \[
        \mu(U)
        \le
        \frac{\rho^2}{(\lambda-\lambda_U)^2+\rho^2}
        \le
        \frac{m_{UW}}{(\lambda-\lambda_U)^2+m_{UW}}.
  \]
\end{enumerate}
\end{lemma}

\begin{proof}
Order the vertices so that the vertices of \(U\) come first and the
vertices of \(W\) come last.  Then
\[
  A(H)=
  \begin{pmatrix}
    A(H[U]) & M\\
    M^{\mathsf T} & A(H[W])
  \end{pmatrix}.
\]

\smallskip
\noindent\textup{(a).}
We write \(A(H)=D+C\), where \(D\) is block diagonal with diagonal blocks
\(A(H[U])\) and \(A(H[W])\), and \(C\) has off-diagonal blocks \(M\) and
\(M^{\mathsf T}\).  Then
\[
        \|D\|=\max\{\lambda_U,\lambda_W\}.
\]
Also,
\[
        C^{\mathsf T}C
        =
        \begin{pmatrix}
          MM^{\mathsf T} & 0\\
          0 & M^{\mathsf T}M
        \end{pmatrix},
\]
and hence \(\|C\|=\|M\|=\rho\).  By the triangle inequality for the
operator norm,
\[
        \lambda=\|A(H)\|
        \le
        \|D\|+\|C\|
        =
        \max\{\lambda_U,\lambda_W\}+\rho.
\]
Finally, \(\rho\le \|M\|_{\mathrm F}=\sqrt{m_{UW}}\), proving (a).

\smallskip
\noindent\textup{(b).}
By monotonicity, we have \(\lambda_U,\lambda_W\le \lambda\).  If
\(\lambda=\lambda_U\) or \(\lambda=\lambda_W\), then the desired
inequality is immediate.  Thus we may assume that $\lambda>\lambda_U$ and $\lambda>\lambda_W.$
Let $\bm{x}=
        \binom{\bm{u}}{\bm{v}}$
be a unit Perron eigenvector of \(H\), where
\(\bm{u}\in\mathbb R^U\) and \(\bm{v}\in\mathbb R^W\).  Under the strict
inequalities above, both \(\bm{u}\) and \(\bm{v}\) are nonzero.  From
\(A(H)\bm{x}=\lambda\bm{x}\), we have
\[
        (\lambda I-A(H[U]))\bm{u}=M\bm{v}.
\]
Therefore $\bm{u}^{\mathsf T}(\lambda I-A(H[U]))\bm{u}
        =
        \bm{u}^{\mathsf T}M\bm{v}.$
Since $\bm{u}^{\mathsf T}A(H[U])\bm{u}
        \le
        \lambda_U\|\bm{u}\|_2^2,$
we obtain
\[
        (\lambda-\lambda_U)\|\bm{u}\|_2^2
        \le
        \bm{u}^{\mathsf T}M\bm{v}
        \le
        \|\bm{u}\|_2\,\|M\bm{v}\|_2
        \le
        \rho\|\bm{u}\|_2\|\bm{v}\|_2.
\]
Dividing by \(\|\bm{u}\|_2\) gives
\begin{equation}
\label{eq-lambda-lambdaU}
        (\lambda-\lambda_U)\|\bm{u}\|_2
        \le
        \rho\|\bm{v}\|_2.
\end{equation}
Similarly, from
$(\lambda I-A(H[W]))\bm{v}=M^{\mathsf T}\bm{u},$
we get
\begin{equation}
\label{eq-lambda-W}
        (\lambda-\lambda_W)\|\bm{v}\|_2
        \le
        \rho\|\bm{u}\|_2.
\end{equation}
Multiplying \eqref{eq-lambda-lambdaU} and \eqref{eq-lambda-W} yields $(\lambda-\lambda_U)(\lambda-\lambda_W)\le \rho^2.$
The second inequality follows from $\rho^2=\|M\|^2\le \|M\|_{\mathrm F}^2=m_{UW}$.
This proves (b).

\smallskip
\noindent\textup{(c).}
If \(\lambda=\lambda_U\), then the hypothesis
\((\lambda-\lambda_U)^2+\rho^2>0\) forces \(\rho>0\), and the first
right-hand side is equal to \(1\).  Hence the desired inequality
\(\mu(U)\le 1\) is trivial.

We may therefore assume that \(\lambda>\lambda_U\).  With the notation
from the proof of part \textup{(b)}, squaring \eqref{eq-lambda-lambdaU}
and using $\|\bm{u}\|_2^2=\mu(U)$ and  $\|\bm{v}\|_2^2=1-\mu(U),$
gives $(\lambda-\lambda_U)^2\mu(U)
        \le
        \rho^2(1-\mu(U)).$ Rearranging,
\[
        \mu(U)
        \le
        \frac{\rho^2}{(\lambda-\lambda_U)^2+\rho^2}.
\]
Since the function $z\mapsto \frac{z}{(\lambda-\lambda_U)^2+z}$ is increasing for \(z\ge 0\), and since \(\rho^2\le m_{UW}\), we also get
\[
        \frac{\rho^2}{(\lambda-\lambda_U)^2+\rho^2}
        \le
        \frac{m_{UW}}{(\lambda-\lambda_U)^2+m_{UW}}.
\]
This proves (c).
\end{proof}

\subsection{Increment on spectral radius of split graph}

We first give the asymptotic of $\lambda(S_{k,m})$.

\begin{lemma}\label{lem:Skm-asymptotic-lambda}
Fix $k\ge 1$ and $m\ge \binom{k}{2}$. In the divisible case $m-\binom{k}{2}=kq$, we have
\begin{equation}\label{eq:Skm-lambda-divisible}
  \lambda(S_{k,m})
  =
  \frac{k-1+\sqrt{4m-k^2+1}}{2}.
\end{equation} 
Moreover, in the general case, we have
\begin{equation}\label{eq:Skm-lambda-asymptotic}
  \lambda(S_{k,m}) = \sqrt m+\frac{k-1}{2}+O_k(m^{-1/2}).
\end{equation}
\end{lemma}

\begin{proof}
We write $m-\binom{k}{2}=kq+r$ with integers \(q\ge0\) and \(0\le r\le k-1\).
If $r=0$, i.e., in the divisible case, then $S_{k,m}=K_k\vee qK_1$.
If \(q=0\), then \(S_{k,m}=K_k\), and the formula is immediate. Thus assume \(q\ge1\).
Let $V(S_{k,m}):=X\cup Y$ be a vertex partition, where $X$ consists of the vertices of the clique $K_k$, and $Y$ consists of the vertices of the independent set $qK_1$.  
Note that the Perron eigenvector $\bm{x} =(x_v)_{v\in V(S_{k,m})}$ has the same coordinates on the vertices of $X$, and also on the vertices of \(Y\). We may assume that $x_v=a$ for all $v\in X$,  and $x_u=b$ for all $u\in Y$. 
Then $\lambda a=(k-1)a+qb$ and $\lambda b=ka$. Thus, $\lambda (S_{k,m})$ is the largest eigenvalue of 
\[
  \begin{pmatrix}
    k-1 & q \\ k & 0 
  \end{pmatrix}.
\]
So $\lambda(S_{k,m})$ is the largest positive root of $x^2-(k-1)x-kq=0$,
which yields  
\[
  \lambda (S_{k,m})
  =
  \frac{k-1 +\sqrt{(k-1)^2 + 4kq}}{2}.
\]
This is the exact formula in \eqref{eq:Skm-lambda-divisible} since $m=\binom{k}{2}+kq$.

Since $\sqrt{4m-k^2+1}=2\sqrt m+O_k(m^{-1/2})$,  (\ref{eq:Skm-lambda-divisible}) implies \eqref{eq:Skm-lambda-asymptotic} in the divisible case.  
In the general case where $1\le r \le k-1$, set $m_-:=m-r$ and $m_+:=m+k-r$. Then $S_{k,m_-}\subseteq S_{k,m}\subseteq S_{k,m_+}$. 
The monotonicity of the spectral radius implies $\lambda(S_{k,m_-})
   \le
  \lambda(S_{k,m})
   \le
  \lambda(S_{k,m_+}).$
Note that the split graphs $S_{k,m_-}$ and $S_{k,m_+}$ are in the divisible case. Thus, applying \eqref{eq:Skm-lambda-divisible} twice, and using $\sqrt{m_\pm}=\sqrt m+O_k(m^{-1/2})$,  
we obtain the desired bound in \eqref{eq:Skm-lambda-asymptotic}.
\end{proof}

The following discrete increment estimate on $\lambda (S_{k,m})$ will be used later.

\begin{lemma} 
\label{lem:Skm-increment}
Fix $k\ge 1$ and $m\ge \binom{k}{2}+1$.
Then
$\lambda(S_{k,m})-\lambda(S_{k,m-1})
   \ge
  \frac{1}{2k(\sqrt m+k)}.$
Consequently, for every $d\in\{0,1,\dots,m-\binom{k}{2}\}$, we have 
$
  \lambda(S_{k,m-d})
   \le
  \lambda(S_{k,m})-\frac{d}{2k(\sqrt m+k)}.$
\end{lemma}

\begin{proof}
We first handle the case \(k=1\). Then \(S_{1,m}\) is the star with \(m\) edges, so $\lambda(S_{1,m})=\sqrt m$. Hence
\[
        \lambda(S_{1,m})-\lambda(S_{1,m-1})
        =
        \sqrt m-\sqrt{m-1}
        =
        \frac{1}{\sqrt m+\sqrt{m-1}}
        \ge
        \frac{1}{2(\sqrt m+1)}.
\]
The corresponding \(d\)-step bound follows by summing this one-step estimate. Thus we may assume \(k\ge2\).

Let $F:=S_{k,m-1}$ and $F^+:=S_{k,m}$, and we write $\lambda_0:=\lambda(F)$.
Let $\bm{x}$ be a unit Perron eigenvector of $F$.
Let $C\subseteq V(F)$ be the $k$-clique and $I:=V(F)\setminus C$, which is independent. We  write $\bm{x}={\bm{x}_C \choose \bm{x}_I}$.

We first show that $\|\bm{x}_C\|_2^2\ge \tfrac12$.
Indeed, in the vertex order $C,I$, we have 
\[
  A(F)=
  \begin{pmatrix}
    J_k-I_k & B\\
    B^{\mathsf T} & 0
  \end{pmatrix},
\]
where $B$ is the $C\times I$ incidence matrix.
From $B^{\mathsf T} \bm{x}_C=\lambda_0 \bm{x}_I$, we get $\bm{x}_C^{\mathsf T}B \bm{x}_I=\lambda_0\| \bm{x}_I\|_2^2.$
Thus
\[
  \lambda_0
  =\bm{x}^{\mathsf T}A(F) \bm{x}
  = \bm{x}_C^{\mathsf T}(J_k-I_k) \bm{x}_C+ 2 \bm{x}_C^{\mathsf T}B \bm{x}_I 
  =
  \bm{x}_C^{\mathsf T}(J_k-I_k) \bm{x}_C+2\lambda_0\| \bm{x}_I\|_2^2.
\]
Since $\bm{x}_C^{\mathsf T}(J_k-I_k)\bm{x}_C\ge 0$, we have $\|\bm{x}_I\|_2^2\le \tfrac12$ and hence $\|\bm{x}_C\|_2^2\ge \tfrac12$.

Let $a:=\max_{v\in C} \{x_v\}$.
Then
\begin{equation}\label{eq:Skm-increment-a2}
  a^2\ \ge\ \frac{\|\bm{x}_C\|_2^2}{k}\ \ge\ \frac{1}{2k}.
\end{equation} 
We also use the crude bound
\begin{equation}\label{eq:Skm-increment-lambda0}
  \lambda_0+1\ \le\ \sqrt m+k.
\end{equation}
To see this, we write $A(F)$ in the same block form above, and let ${\bm{u} \choose \bm{v}}$ be a unit vector with $\bm{u}\in\mathbb R^C$ and $\bm{v} \in\mathbb R^I$.
Denote $\|\bm{u}\|_2=\alpha$ and $\|\bm{v}\|_2=\beta$, we have $\alpha^2+\beta^2=1$. 
Using the Rayleigh quotient, we get
\[
        \bm{u}^{\mathsf T}(J_k -I_k) \bm{u}
        \le
        (k-1) \|\bm{u}\|_2^2.
\]
Using Cauchy--Schwarz gives 
\[
  (\bm{u}^{\mathsf T},\bm{v}^{\mathsf T})A(F)
  \begin{pmatrix}
    \bm{u} \\ \bm{v}
  \end{pmatrix}
  \le
  (k-1)\alpha^2+2\|B\|\alpha\beta.
\]
Since $ \|B\|^2 \le \|B\|_{\mathrm F}^2=e_F(C,I)=m-1-\binom{k}{2}\le m$, maximizing over $\alpha^2+\beta^2=1$ gives
\[
  \lambda_0
  \ \le\
  \frac{k-1+\sqrt{(k-1)^2+4m}}{2}
  \ \le\
  \sqrt m+k-1,
\]
which implies \eqref{eq:Skm-increment-lambda0}.

We write $m-\binom{k}{2}=kq+r$ with $0\le r\le k-1$.
If $r=1$, then $F=K_k\vee qK_1$ and $F^+$ is obtained from $F$ by adding a new vertex $w$ adjacent to one clique vertex $u$.
By symmetry of $F$, all clique vertices have the same Perron coordinate, so $x_u=a$.
For $\alpha\ge 0$, consider the vector $\bm{y}:={\bm{x} \choose \alpha}$ on $V(F^+)$.
Then $\| \bm{y}\|_2^2=1+\alpha^2$ and
\[
        \bm{y}^{\mathsf T}A(F^+) \bm{y}=\lambda_0+2a\alpha.
\]
Applying the Rayleigh quotient,
\[
  \lambda(F^+)
  \ \ge\
  \max_{\alpha\ge 0}\frac{\lambda_0+2a\alpha}{1+\alpha^2}
  =
  \frac{\lambda_0+\sqrt{\lambda_0^2+4a^2}}{2}.
\]
Thus
\[
  \lambda(F^+)-\lambda_0
  \ge
  \frac{2a^2}{\sqrt{\lambda_0^2+4a^2}+\lambda_0}
  \ge
  \frac{a^2}{\lambda_0+1},
\]
where the last inequality holds because $\sqrt{\lambda_0^2+4a^2}\le \lambda_0+2$.
Using \eqref{eq:Skm-increment-a2} and \eqref{eq:Skm-increment-lambda0} gives the desired bound.

If $r\neq 1$, then $F^+$ is obtained from $F$ by adding one edge $wv$, where $w$ is a vertex of the independent set of $F$
and $v$ is a vertex of the clique $K_k$.
Let $s:=\deg_F(w)\in\{1,\dots,k-1\}$. 
Recall that $a=\max_{v\in C}x_v$. 
By symmetry, we may assume that $\bm{x}$ takes value $a$ on the $s$ neighbors of $w$, and value $b$ on the remaining vertices of $K_k$.  
Write $d:=x_w$.
The eigenvector equation at \(w\) gives
\[
        \lambda_0 d=sa,
        \qquad\text{that is,}\qquad
        d=\frac{sa}{\lambda_0}.
\]
In addition, comparing the eigenvector equations at a neighbor and a non-neighbor of $w$ gives 
\[
        \lambda_0 a - \lambda_0 b = b-a+d,
\]
so  $(\lambda_0+1)(a-b)=d.$
Thus
\[
  b
  =
  a-\frac{d}{\lambda_0+1}
  =
  a\left(1-\frac{s}{\lambda_0(\lambda_0+1)}\right)
  \ge
  \left(1-\frac1k \right)a
  \ge
  \frac{a}{2},
\]
where we used $\lambda_0\ge k-1 \ge s$.
Therefore $bd\ge \frac{sa^2}{2\lambda_0}$. By the Rayleigh quotient,
\[
  \lambda(F^+)-\lambda_0
  \ge
  2bd
  \ge
  \frac{sa^2}{\lambda_0}
  >
  \frac{a^2}{\lambda_0+1}.
\]
Combining \eqref{eq:Skm-increment-a2} and \eqref{eq:Skm-increment-lambda0} gives the desired bound. 

The bound for $\lambda(S_{k,m-d})$ follows by summing the one-step estimate for $i=m-d+1,\dots,m$ and using $\sqrt i+k\le \sqrt m+k$. This completes the proof.
\end{proof}

\subsection{Heavy-edge subgraph with the split gap}

We employ an edge-deletion technique that removes edges whose endpoints have small product of Perron weights. 
This technique has its roots in Nikiforov  \cite{Nikiforov2007C4FreeMaxSpectralRadius}, 
and was later developed by Nikiforov \cite{Niki2021} to solve a spectral problem related to books; it was subsequently applied by Ning--Zhai \cite{NZ2021b} to study spectral supersaturation for \(4\)-cycles. 
We start with the following definition.

\begin{definition} \label{def:heavy-edge}
Let $H$ be an $m$-edge graph and let $\bm{x}=(x_v)_{v\in V(H)}$ be a unit Perron eigenvector of $H$.
For $\eta>0$, we say that an edge $uv\in E(H)$ is \emph{$\eta$-heavy} if
\[
        x_u x_v \ge \frac{\eta}{\sqrt m}.
\]
We say that $H$ satisfies the \emph{$\eta$-heavy-edge condition} if every edge of $H$ is $\eta$-heavy. 
\end{definition}

\begin{lemma} \label{lem:lemmaA-pruning}
Fix $t\ge 2$ and $\eta:=1/(16t)$.
For all sufficiently large $m$, if $G$ is an $m$-edge graph with $\lambda(G) > \lambda(S_{t-1,m})$, then $G$ contains a subgraph $H\subseteq G$ with
$m':=e(H)$ and a unit Perron eigenvector $\bm{x}$ such that $H$ satisfies the $\eta$-heavy-edge condition and $\lambda(H)>\lambda(S_{t-1,m'}).$
Moreover, $\frac{\lambda(H)}{\sqrt{m'}}
        \ge
        1+(1-4\eta)\bigl(\alpha^{-1/2}-1\bigr)$, where $\alpha:=\frac{m'}{m}.$
\end{lemma}

\begin{proof} 
We construct a sequence of graphs $G:=G_0\supseteq G_1\supseteq \cdots \supseteq G_\ell :=H$ iteratively as follows.
Given $G_i$, let $\bm{x}^{(i)}$ be a unit Perron eigenvector and write $m_i:=e(G_i)$.
If there is an edge $u_iv_i\in E(G_i)$ with $x^{(i)}_{u_i}x^{(i)}_{v_i}<\frac{\eta}{\sqrt{m_i}},$
delete one such edge to obtain $G_{i+1}$; otherwise, stop the deletion. 
In Claim~\ref{cl:H-linearly-size} below, we show that this process terminates at a subgraph \(H\) with linearly many edges.  
For each $0\le i\le \ell$, 
we write $\lambda_i:=\lambda(G_i)$, $m:=m_0=e(G)$, and $m':=m_\ell=e(H)$.
Each deletion step satisfies $m_{i+1}=m_i-1$. 
Let $\bm{x}:=\bm{x}^{(\ell)}$ be a unit Perron eigenvector of~$H$. 
At the final step $\ell$, by definition there is no edge $uv\in E(G_\ell)$ with
$x^{(\ell)}_ux^{(\ell)}_v<\frac{\eta}{\sqrt{m_\ell}}.$
Thus, every edge $uv\in E(H)$ satisfies $x_u x_v\ge \eta/\sqrt{m'}$.

\begin{claim}\label{clm:pruning-one-step-rayleigh}
For each $0\le i\le \ell-1$, one has $\lambda_{i+1} \ge \lambda_i-\frac{2\eta}{\sqrt{m_i}}.$
\end{claim}

\begin{poc} 
For each $i$, write $G_{i+1}=G_i-u_iv_i$ and set $A_i:=A(G_i)$ and $A_{i+1}:=A(G_{i+1})$.
Since $\bm{x}^{(i)}$ is a unit Perron eigenvector of $A_i$ corresponding to \(\lambda_i\), we have $\lambda_i=(\bm{x}^{(i)})^{\mathsf T}A_i \bm{x}^{(i)}.$
By the Rayleigh quotient applied to $A_{i+1}$, we get 
\[
  \lambda_{i+1}
  \ \ge\
  (\bm{x}^{(i)})^{\mathsf T}A_{i+1}\bm{x}^{(i)}
  =
  (\bm{x}^{(i)})^{\mathsf T}A_i \bm{x}^{(i)}-2x^{(i)}_{u_i}x^{(i)}_{v_i}
  =
  \lambda_i-2x^{(i)}_{u_i}x^{(i)}_{v_i}
  \ \ge\ \lambda_i-\frac{2\eta}{\sqrt{m_i}}. 
  \qedhere 
\]
\end{poc}
 
Now we show that edge loss in the deletion process forces a multiplicative spectral gap.

\begin{claim} \label{claim:multiplicative-gap}
  We have
$\frac{\lambda(H)}{\sqrt{m'}}
        \ge
        1+ (1- 4\eta) (\alpha^{-1/2} -1).$
\end{claim}

\begin{poc}
By Lemma~\ref{lem:Skm-asymptotic-lambda} with \(k=t-1\), $\lambda(S_{t-1,m})=\sqrt m+\frac{t-2}{2}+O_t(m^{-1/2}),$
 so $\lambda(S_{t-1,m})\ge \sqrt m$
for all sufficiently large~$m$.
Increasing the lower bound on \(m\) if necessary, the hypothesis
\(\lambda(G)>\lambda(S_{t-1,m})\) therefore implies \(\lambda(G)\ge \sqrt m\).
For each deletion step, we have $m_{i+1}=m_i-1$ and 
\[
  \sqrt{m_i}-\sqrt{m_{i+1}}
  =
  \frac{1}{\sqrt{m_i}+\sqrt{m_{i+1}}}
  \ \ge\ \frac{1}{2\sqrt{m_i}},
  \qquad\text{i.e.}\qquad
  \frac{1}{\sqrt{m_i}}\ \le\ 2(\sqrt{m_i}-\sqrt{m_{i+1}}).
\]
Combining with Claim~\ref{clm:pruning-one-step-rayleigh} gives
\[
  \lambda_{i+1}-\sqrt{m_{i+1}}
  \ \ge\
  (\lambda_i-\sqrt{m_i})+(1-4\eta)\bigl(\sqrt{m_i}-\sqrt{m_{i+1}}\bigr).
\]
Summing over $i=0,1,\dots,\ell-1$ and telescoping yields
\[
  \lambda(H)-\sqrt{m'}
  \ \ge\
  \lambda(G)-\sqrt m\;+\;(1-4\eta)\bigl(\sqrt m-\sqrt{m'}\bigr).
\]
This also shows that \(m'>0\). Indeed, if \(m'=0\), then \(H\) is edgeless, so \(\lambda(H)=0\), whereas the right-hand side above is at least \((1-4\eta)\sqrt m>0\), a contradiction.
Using $\lambda(G)\ge \sqrt m$ and dividing by $\sqrt{m'}$, we obtain
\[
  \frac{\lambda(H)}{\sqrt{m'}}
  \ \ge\
  1+(1-4\eta)\left(\frac{\sqrt m}{\sqrt{m'}}-1\right)
  =
  1+(1-4\eta)\left(\frac{1}{\sqrt\alpha}-1\right). \qedhere 
\]
\end{poc}

Next, we give a lower bound on the number of remaining edges of $H$.

\begin{claim} \label{cl:H-linearly-size}
The resulting graph $H$ contains $m' \ge c_{\eta}m$ edges, where        $c_{\eta}:=\left(\frac{1-4\eta}{\sqrt{2} - 4\eta}\right)^2>0.$
\end{claim}

\begin{poc}
Since $\lambda(H)\le \sqrt{2m'}$,
Claim~\ref{claim:multiplicative-gap} gives  
$1 + (1- 4\eta ) (\alpha^{-1/2} - 1) \le \sqrt{2}.$
Thus $\alpha=\frac{m'}{m}\ge c_{\eta}.$
\end{poc}
 
It remains to show that the pruning preserves the split gap, that is, $\lambda(H)>\lambda(S_{t-1,m'}).$
If \(\ell=0\), this is immediate from the hypothesis. Hence assume \(\ell\ge1\).
Define
\[
        \Delta_i := \lambda_i- \lambda(S_{t-1,m_i})
\]
for every $0\le i\le \ell$. 
In particular, $\Delta_0=\lambda(G)- \lambda(S_{t-1,m})>0$ by assumption. 
To prove $\lambda(H) > \lambda (S_{t-1,m'})$, it suffices to show that $\Delta_{i+1} \ge \Delta_i$ for every \(0\le i\le \ell-1\).

Since $\eta= \frac{1}{16t}$, we can choose $M=M(t,\eta)$ such that for all \(s\ge M\),
\[
\frac{1}{2(t-1)(\sqrt s+t-1)} > \frac{2\eta}{\sqrt s}.
\]
Then for all sufficiently large $m$, the lower bound $m'\ge c_\eta m$ in Claim~\ref{cl:H-linearly-size} implies $m_i\ge m'\ge M$ for each $0\le i\le \ell-1$. 
By Claim~\ref{clm:pruning-one-step-rayleigh}, $\lambda_{i+1} - \lambda_i \ge -\frac{2\eta}{\sqrt{m_i}}.$
By Lemma~\ref{lem:Skm-increment}, applied with \(k=t-1\), we get
\[
  \lambda(S_{t-1,m_i})-
  \lambda(S_{t-1,m_i-1})
  \ge
  \frac{1}{2(t-1)(\sqrt{m_i}+t-1)}.
\]
Combining the above inequalities,
\[
  \Delta_{i+1} - \Delta_i
  =
  \lambda_{i+1} - \lambda_i
  +
  \lambda(S_{t-1,m_i})
  -
  \lambda(S_{t-1,m_i-1})
  >0.
\]
Hence $\Delta_{i+1}> \Delta_i$ for all \(0\le i\le \ell-1\). Therefore $\Delta_\ell>\Delta_{\ell-1}>\cdots>\Delta_0>0,$
and so $\lambda(H)>\lambda(S_{t-1,m'}).$
This completes the proof of Lemma~\ref{lem:lemmaA-pruning}.
\end{proof}

\section{Proof of Theorem \ref{thm:sharp-Ktt}} 

\label{sec-four-key-lemmas}

First, we show that the constant $B_t= 2^{-(t-1)^2}/(t!)^2$ in Theorem \ref{thm:sharp-Ktt} is best possible. 

\begin{example} \label{exam:sharpness-Bt}
For every integer $t\ge 2$, 
   the random graphs $G(n,m)$ on $n\approx 2\sqrt{m} - t$ vertices with $m$ edges achieve the sharp constant $B_t$ in Theorem \ref{thm:sharp-Ktt}. 
\end{example}

\begin{proof} 
Fix $t\ge 2$ and let $m\to\infty$.
Let $n:=\lfloor 2\sqrt m-t\rfloor$ and $N:=\binom n2$.
For $m$ sufficiently large, one has $m\le N$, so the random graph model $G(n,m)$ is well-defined\footnote{The model $G(n,m)$ selects uniformly at random a graph with $n$ vertices and $m$ edges from all $\binom{N}{m}$ such graphs.}.
Every $m$-edge graph $G$ on $n$ vertices satisfies $\lambda(G)\ge \frac{2m}{n}$. Note that  
\[
  \frac{2m}{n}
  \ \ge\
  \frac{2m}{2\sqrt m-t}
  \ =\
  \sqrt m+\frac{t}{2}+\frac{t^2}{4\sqrt m-2t}
  \ > \
  \sqrt m+\frac{t-2}{2}+1.
\]
Since $\lambda(S_{t-1,m})=\sqrt m+\frac{t-2}{2}+O_t(m^{-1/2})$ by Lemma \ref{lem:Skm-asymptotic-lambda}, this implies $\lambda(G)>\lambda(S_{t-1,m})$ for all sufficiently
large $m$.
If $G$ is a random graph in the model $G(n,m)$, then the expectation 
\[
\mathbb E[\texttt{\#}\Ktt(G)]
  = 
  \frac12\binom nt\binom{n-t}{t}
  \cdot
  {\binom{N-t^2}{m-t^2}}\bigg/ {\binom{N}{m}}. 
\]
Using $\binom nt\binom{n-t}{t}=\frac{n^{2t}}{(t!)^2}\left(1+O_t(n^{-1})\right)$ and 
\[
  {\binom{N-t^2}{m-t^2}} \bigg/ {\binom{N}{m}}
  \;=\;
  \prod_{i=0}^{t^2-1}\frac{m-i}{N-i}
  \;=\;
  \left(\frac{m}{N}\right)^{t^2}\left(1+O_t(m^{-1})\right), 
\]
which together with $n=2\sqrt m+O_t(1)$ and
$N=2m+O_t(\sqrt m)$ yields 
\[
 \mathbb E[\texttt{\#}\Ktt(G)]
  =
  \left( \frac{2^{-(t-1)^2}}{(t!)^2}+O_t(m^{-1/2})\right)m^t. 
\]  
Hence, for every $\varepsilon>0$ and all sufficiently large $m$, there exists an $m$-edge graph $G$ with
$\lambda(G)>\lambda(S_{t-1,m})$, while $\texttt{\#}\Ktt(G)\le (B_t+\varepsilon)m^t$. This shows the tightness of $B_t$. 
\end{proof}

We point out that for every $t\ge 3$, the graphs near the split graph $S_{t,m}$ are not sharp for Theorem~\ref{thm:sharp-Ktt}.
This is different from the case $t=2$ in Theorem \ref{thm:counting-C4}.

\subsection{Delocalized case}

\begin{lemma} \label{lem:lemmaB-delocalized}
Fix \(t\ge2\) and $\eta\in(0,\tfrac14)$. Let $H$ be an $m$-edge graph with unit Perron eigenvector \(\bm{x}\). Denote
$g:=\|\bm{x} \|_\infty m^{1/4}$. 
If \(H\) satisfies the \(\eta\)-heavy-edge condition and \(g=O(1)\), then
\[
 \texttt{\#}\Ktt(H)\ge \bigl(B_t-o(1)\bigr) \Big(\frac{\lambda^2(H)}{m}\Big)^{t(t-1)}m^t.
\]
\end{lemma}

In the proof of Lemma \ref{lem:lemmaB-delocalized}, we use a trick of the authors \cite{LiLiuZhang2025EdgeSpectralESS}. 
Our goal is to show that $H$ is dense, i.e.~supported on \(O(\sqrt m)\) vertices, so the term $O_t(n^{2t-1})$ in the bound of Lemma \ref{lem:spectral-count-Ktt} is negligible. 

\begin{proof} 
Fix $t\ge 2$ and $\eta\in(0,\tfrac14)$. 
We denote $\lambda : =\lambda (H)$. 
Let $H$ be an $m$-edge graph with unit Perron eigenvector $\bm{x}$ satisfying the $\eta$-heavy-edge condition. 
Delete all isolated vertices of \(H\), if any. This does not change \(m\), \(\lambda(H)\), \(\texttt{\#}\Ktt(H)\), the heavy-edge condition, or the value of \(g\). Thus we may assume that \(H\) has no isolated vertices. Let \(n:=|V(H)|\).
Assume that $g=\|\bm{x}\|_\infty m^{1/4}=O(1)$, and fix $G_0>0$ such that $g\le G_0$ for all sufficiently large $m$.
Then $\| \bm{x}\|_\infty\le G_0 m^{-1/4}$. 
For any vertex $v\in V (H)$, choose a neighbor $u\in N(v)$; then by the heavy-edge condition,
\[
  x_v\ \ge\ \frac{\eta}{x_u\sqrt m}\ \ge\ \frac{\eta}{\| \bm{x}\|_\infty\sqrt m}\ \ge\ \frac{\eta}{G_0}m^{-1/4}.
\]
Using $\| \bm{x}\|_2=1$ gives  $1 \ge n \cdot {\eta}^2{G_0}^{-2} m^{-1/2},$ so $
 n \le {G_0^2}{\eta^{-2}}\sqrt m.$
Lemma~\ref{lem:spectral-count-Ktt} yields
\[
  \texttt{\#}\Ktt(H)
  \ \ge\
  B_t\left(\frac{\lambda^2}{m}\right)^{t(t-1)}m^t
  \ -\ O_t\!\left(n ^{2t-1}\right).
\]
Since $n=O(\sqrt m)$, the error term $O_t(n^{2t-1})= o(m^t)$.
Also the $\eta$-heavy-edge condition implies
\[ 
  \lambda =\bm{x}^{\mathsf T}A(H) \bm{x}
  =2\sum_{uv\in E(H)}x_u x_v
  \ge 2m\cdot \frac{\eta}{\sqrt m}
  =2\eta\sqrt m.
\] 
So the leading term is $\left(\frac{\lambda^2}{m}\right)^{t(t-1)}m^t=\Omega_{t,\eta}(m^t).$
Thus, the error term \(O_t(n^{2t-1})=o(m^t)\) can be absorbed into the $(B_t-o(1))$
coefficient. 
This gives the desired bound.
\end{proof}

\subsection{A structural result for the localized case}

In this section, we first establish a structural result. Roughly, it says that if $H$ is an $m$-edge graph satisfying both the $\eta$-heavy-edge condition and $g:=\|\bm{x}\|_{\infty} m^{1/4} \to \infty$,
then $H$ admits a useful vertex partition. 
To be more precise, we introduce the following definition. 

\begin{definition} \label{def:ACD-partition} 
We say that a partition $V(H)=A\sqcup C\sqcup D$ satisfies conditions (T1)--(T3) if \textbf{(T1)} $D$ is independent;
\textbf{(T2)} $E_H(C,D)=\varnothing$;
\textbf{(T3)} every edge of $H$ lies either in $E(H[C])$ or is incident to~$A$. 
We say that such a partition is an \emph{ACD partition}, and we call \(C\) the \emph{core}. 
\end{definition}

We point out that the ACD partition is quite different from the traditional bipartition obtained from the edge-spectral stability theorem recently established by the authors \cite{LiLiuZhang2025EdgeSpectralESS}. 

In the localized case, we first prove the following structural result. 

\begin{theorem} 
\label{thm:localized-structure}
Let $H$ be an $m$-edge graph with unit Perron eigenvector $\bm{x}$. 
Assume that $H$ satisfies the $\eta$-heavy-edge condition for some fixed $\eta\in(0,\tfrac14)$ and $g=\|\bm{x} \|_\infty m^{1/4} \to\infty$. 
Then $H$ admits an ACD partition $V(H)=A\sqcup C\sqcup D$ such that 
\[
 e_H(A,C)=o(m), \quad 
|A|=o(\sqrt{m}), \quad 
|C| \le \sqrt{m}\cdot m^{o(1)}. 
\] 
\end{theorem}

\begin{proof}  
Write $L:=\|\bm{x}\|_\infty$ and $g:=Lm^{1/4}$.
Assume $g=g(m)\to\infty$ as $m\to\infty$ and the $\eta$-heavy-edge condition
$ x_u x_v\ \ge\ \frac{\eta}{\sqrt m}$ for every edge $uv\in E(H)$.  
We begin by producing an ACD partition. 
Let $\widetilde g:=g/\sqrt\eta$ and $ K:=\left\lceil 2\log_2 \widetilde g\right\rceil+1$. 
For \(0\le h\le K\), set \(\theta_h:=2^{-h}L\).
Let
\[
  I:=\Bigl\{\lfloor K/2\rfloor-\lfloor\sqrt K\rfloor,\ \dots,\ \lfloor K/2\rfloor-\lceil\sqrt K/2\rceil\Bigr\}.
\]
For $i\in\{1,2,\dots,\lfloor K/2\rfloor\}$, we define the intermediate level sets
\[
  C_i:=\{v:\ \theta_i\ge x_v>\theta_{K-i}\},
  \qquad
  B_i:=\{v:\ \theta_{i-1}\ge x_v>\theta_i\},
  \qquad
  F_i:=E_H(C_i,B_i).
\]
The sets $B_1,\dots,B_{\lfloor K/2\rfloor}$ are pairwise disjoint, so the edge sets $F_i$ are disjoint and hence
$\sum_{i=1}^{\lfloor K/2\rfloor}|F_i|\le m$. 
Let $\ell:=\lceil \log_2 K\rceil$.
For each integer $i\in I$ with $i\ge \ell$, we define $S_i:=\sum_{j=0}^{\ell-1}|F_{i-j}|$.
Note that 
\[
 \sum_{{i\in I, i\ge \ell}} S_i \le \ell m,
\]
since each $|F_j|$ appears in at most $\ell$ different sums $S_i$. 
Since $g\to\infty$, we have $\widetilde{g} \to \infty$ and $K\to\infty$. Then $\ell=o(\sqrt K)=o(|I|)$ as $|I| \approx \sqrt{K}/2$. 
For all sufficiently large \(m\), every \(i\in I\) satisfies
$i\ge \lfloor K/2\rfloor-\lfloor\sqrt K\rfloor \ge \ell,$
so the quantities \(S_i\) are defined for all \(i\in I\).
By the pigeonhole principle, 
there exists $i_\star\in I$ with
\[
 S_{i_\star} \le \frac{1}{|I|} \sum_{i\in I} S_i \le \frac{\ell m}{|I|}= o(m).
\]
Denote $s:=\theta_{i_\star}$ and $r:=\theta_{K-i_\star}$. Then $s>r >0$. We define $A,C,D$ as follows:
\[
  A:=\{v \in V(H):\ x_v>s\},\quad
  C:=\{v\in V(H):\ s\ge x_v>r\},\quad
  D:=\{v\in V(H):\ x_v\le r\}.
\] 

\begin{claim}
We have \(e_H(A,C)=o(m)\).
\end{claim}

\begin{poc}
To begin with, we partition $A=A_{\mathrm{mid}}\sqcup A_{\mathrm{very}}$, where 
\[
  A_{\mathrm{mid}}:=\{v:\ s<x_v<2^\ell s\},
  \qquad
  A_{\mathrm{very}}:=\{v:\ x_v\ge 2^\ell s\}. 
\]
Since $2^\ell s=\theta_{i_\star-\ell}$, we have
$A_{\mathrm{mid}}\subseteq B_{i_\star}\sqcup\cdots\sqcup B_{i_\star-\ell+1}$.
Also $C=C_{i_\star}\subseteq C_{i_\star-j}$ for $0\le j\le \ell-1$, so
$E_H(C,B_{i_\star-j})\subseteq E_H(C_{i_\star-j},B_{i_\star-j})=F_{i_\star-j}$.
Therefore
\[
  e_H(A_{\mathrm{mid}},C)\ \le\ \sum_{j=0}^{\ell-1}|F_{i_\star-j}|\ =\ S_{i_\star}\ =\ o(m).
\] 
For $u\in C$ and $v\in A_{\mathrm{very}}$ one has $x_u>r$ and $x_v\ge 2^\ell s$, hence $x_u x_v\ge 2^\ell sr$.
Since $sr=\theta_{i_\star}\theta_{K-i_\star}=2^{-K}L^2$ and $2^K<4\widetilde g^{\,2}$, we have
$sr>\eta/(4\sqrt m)$.
Thus each edge $uv\in E_H(A_{\mathrm{very}},C)$ satisfies $x_u x_v>2^\ell\eta/(4\sqrt m)$, and so
\[
  e_H(A_{\mathrm{very}},C)\cdot 2^\ell\frac{\eta}{4\sqrt m}
  \ <\
  \sum_{ab\in E(H)}x_ax_b
  \ =\ \frac{1}{2}\lambda(H) \le \frac{1}{2}\sqrt{2m}.
\]
Using $2^\ell\ge K\to\infty$ gives
$e_H(A_{\mathrm{very}},C)=o(m)$. 
So we conclude that $e_H(A,C)=o(m)$.
\end{poc}

\begin{claim}
    The partition $V(H)=A\sqcup C\sqcup D$ satisfies \textup{(T1)}--\textup{(T3)} of Definition \ref{def:ACD-partition}. 
\end{claim}

\begin{poc}
We first show that $sr<\eta/\sqrt m$.
Indeed, since $2^K\ge 2\widetilde g^{\,2}$, we have
\[
  sr
  =
  2^{-K}L^2
  \le
  \frac{L^2}{2\widetilde g^{\,2}}
  =
  \frac{L^2}{2(g^2/\eta)}
  =
  \frac{\eta}{2\sqrt m}
  <\ \frac{\eta}{\sqrt m}.
\]
If $u\in C\cup D$ and $v\in D$, then $x_u\le s$ and $x_v\le r$, so $x_ux_v\le sr< \frac{\eta}{\sqrt m}$.
By the heavy-edge assumption, there are no $C$--$D$ edges and $D$ is independent, proving \textup{(T1)}--\textup{(T2)}. 
Now, let $uv\in E(H)$ be an edge not incident to $A$.
Then $x_u,x_v\le s$, so neither endpoint can lie in $D$, since otherwise \(x_ux_v<\eta/\sqrt m\), and hence $u,v\in C$.
Thus every edge not incident to $A$ lies in $H[C]$, which is exactly \textup{(T3)}.
\end{poc}

\begin{claim}
    We have $|A|=o(\sqrt m)$ and 
  $|C|\le \sqrt m\cdot m^{o(1)}$. 
\end{claim}

\begin{poc}
Since $\|\bm{x}\|_2=1$ and every vertex in $A$ satisfies $x_v>s$, we have $|A|s^2<1$ and hence $|A|<1/s^2$.
Write $j:=\lfloor K/2\rfloor-i_\star$.
As $i_\star\in I$, we have $j\ge \lceil \sqrt K/2\rceil\to\infty$ and also $j=O(\sqrt K)$.
Moreover $2i_\star\le K-2j$, so
\[
  \frac{1}{s^2}=\frac{2^{2i_\star}}{L^2}\ \le\ \frac{2^K}{L^2}\,2^{-2j}.
\]
Since $K\le 2\log_2 \widetilde g+2$, we have $2^K\le 4\widetilde g^{\,2}$ and therefore
\[
  \frac{2^K}{L^2}\ \le\ \frac{4\widetilde g^{\,2}}{L^2}
  =\frac{4(g^2/\eta)}{L^2}
  =\frac{4\sqrt m}{\eta}.
\]
Combining these bounds gives $|A|\le (4/\eta)\,2^{-2j}\sqrt m=o(\sqrt m)$.

Similarly, every vertex in $C$ satisfies $x_v>r$, so $|C|r^2\le 1$ and hence $|C|\le 1/r^2$.
Since $r=\theta_{K-i_\star}=2^{-(K-i_\star)}L$ and $K-i_\star\le \lceil K/2\rceil+j$, we obtain
\[
  \frac{1}{r^2}=\frac{2^{2(K-i_\star)}}{L^2}\ \le\ \frac{2^{K+2j+2}}{L^2}
  \ \ll_\eta\ 2^{2j}\sqrt m.
\]
As $j=O(\sqrt K)$ and $K=O(\log \widetilde g)=O(\log m)$, we have $2^{2j}=m^{o(1)}$, and hence $|C|\le \sqrt m\cdot m^{o(1)}.$
\end{poc}

The three claims complete the proof.
\end{proof}

\subsection{Localized case with a dense core}

The following lemma deals with the localized case with a dense core.

\begin{lemma} 
\label{lem:lemmaC-localized-nonsparse}
Let $H$ be an $m$-edge graph with the unit Perron eigenvector \(\bm{x}\), and let  
$g:=\|\bm{x} \|_\infty m^{1/4}$.  
Assume that \(H\) satisfies the \(\eta\)-heavy-edge condition for some fixed \(\eta\in(0,\tfrac14)\) and that \(g\to\infty\). Assume in addition that \(\lambda(H)\ge\sqrt m\), and that the partition \(V(H)=A\sqcup C\sqcup D\) obtained from Theorem~\ref{thm:localized-structure} satisfies
$e(H[C])=(\alpha+o(1))m$ for some constant \(\alpha\in(0,1]\). Then
$\texttt{\#}\Ktt(H)\ge (B_t-o(1))m^t$.
\end{lemma}

\begin{proof}  
Set $B:=A\sqcup D$ and write 
$ \lambda:=\lambda(H)$, $ \lambda_C:=\lambda\bigl(H[C]\bigr)$ 
and $ \lambda_B:=\lambda\bigl(H[B]\bigr)$.  
Applying Lemma~\ref{lem:spectral-vertex-cut} \textup{(a)} to the graph $H[B]$ with the partition $B=A\sqcup D$ yields, since \(\lambda(H[D])=0\),
\begin{equation}\label{eq:lemC-lambdaB-upper}
  \lambda_B
  \ \le\
  \lambda\bigl(H[A]\bigr)+\sqrt{e_H(A,D)}
  \ \le\
  |A|+\sqrt{e\bigl(H[B]\bigr)}.
\end{equation} 
Next apply Lemma~\ref{lem:spectral-vertex-cut} \textup{(b)} to the partition $V(H)=C\sqcup B$.
Since \textup{(T2)} holds, the cut between $C$ and $B$ consists only of $A$--$C$ edges, so $e_H(C,B)=e_H(C,A)=o(m)$.
Therefore
\begin{equation}\label{eq:lemC-cut-forcing}
  (\lambda-\lambda_C)(\lambda-\lambda_B)\ \le\ e_H(C,B)\ =\ o(m).
\end{equation}

We now count copies of $\Ktt$ inside the induced subgraph  $H[C]$ by applying Lemma~\ref{lem:spectral-count-Ktt}. 
By \textup{(T1)}--\textup{(T3)},
\[
  m = e (H[C]) + e(H[B]) + e_H(A,C) .
\] 
Using the assumption $e(H[C])=(\alpha+o(1))m$, together with  $e_H(A,C)=o(m)$, we have $e(H[B])=(1-\alpha+o(1))m$. 
Substituting into \eqref{eq:lemC-lambdaB-upper} gives $\lambda_B\ \le\ \bigl(\sqrt{1-\alpha}+o(1)\bigr)\sqrt m,$
and therefore
\[
        \lambda-\lambda_B
        \ge 
        \bigl(1- \sqrt{1-\alpha} - o(1) \bigr) \sqrt{m}
        =
        \Omega(\sqrt m),
\]
as $\alpha >0$. 
Substituting into \eqref{eq:lemC-cut-forcing} yields $\lambda-\lambda_C=o(\sqrt m)$, and hence $\lambda_C\ge (1-o(1))\sqrt m$. 

Let \(m_C:=e(H[C])\).
Applying Lemma~\ref{lem:spectral-count-Ktt} to $H[C]$ gives
\[
  \texttt{\#}\Ktt(H)
  \ \ge\
  \texttt{\#}\Ktt(H[C])
  \ \ge\
  B_t\left(\frac{\lambda_C^2}{m_C}\right)^{t(t-1)}m_C^t\ -\ o(m_C^t)
  \ \ge \
  \bigl(B_t-o(1)\bigr)m^t,
\]
where the second inequality holds since $O_t(|C|^{2t-1})
        \le
        \left(\sqrt m\,m^{o(1)}\right)^{2t-1}
        =
        o(m^t).$
For the last inequality, use \(\lambda_C^2\ge (1-o(1))m\) and \(m_C=(\alpha+o(1))m\):
\[
  \left(\frac{\lambda_C^2}{m_C}\right)^{t(t-1)}m_C^t
  \ge
  (1-o(1))\,\alpha^{\,t-t(t-1)}m^t
  \ge
  (1-o(1))m^t,
\]
because \(\alpha\in(0,1]\) and \(t-t(t-1)=t(2-t)\le0\).
\end{proof}

\subsection{Localized case with a sparse core}

In this section, 
we consider the case where the induced subgraph $H[C]$ is sparse. 

\begin{lemma} \label{lem:lemmaD-localized-sparse} 
Let \(t\ge2\) and $H$ be a graph with $m\to \infty$ edges. 
Assume that $H$ admits a partition $V(H)=A\sqcup C\sqcup D$ satisfying conditions
\textup{(T1)}--\textup{(T3)}. Suppose moreover that
 $e_H(A,C)=o(m)$ and $|A|=o(\sqrt{m})$. 
 If in addition $\lambda(H)>\lambda(S_{t-1,m})$ and $e(H[C])=o(m)$,  then $\texttt{\#}\Ktt(H)\ge \bigl(B_t-o(1)\bigr) m^t$.
\end{lemma}

To prove Lemma \ref{lem:lemmaD-localized-sparse}, we first show that either there are many copies of $K_{t,t}$ using only edges between $A$ and $D$, or $H$ differs from a complete bipartite graph $K_{R,B}$ in $o(m)$ edges, where $|R|\le t-1$; see the following Theorem \ref{thm:sparseC-rowcover}.  
In the latter case, we further show that 
$\lambda (H)< \lambda (S_{t-1,m})$. Thus, this case does not occur under the assumption of Lemma \ref{lem:lemmaD-localized-sparse}; see Subsection \ref{sec:4-5-2}.   

\subsubsection{Rank-one structure of $A$--$D$ incidence matrix}

\label{sec:4-5-1}

We use the following terminology: \(k\)-row covers and \(\theta\)-aligned rows.

\begin{definition} \label{def:row-cover}
Let $M\in\{0,1\}^{A\times D}$ be the $A$--$D$ incidence matrix.
For $k\ge 1$, a set $R\subseteq A$ is a \emph{$k$-row cover} if $|R|\le k$ and $e_H(A\setminus R,\ D)=0$.
Along an asymptotic regime $e(H)\to\infty$, we call $R$ an \emph{asymptotic $k$-row cover} if
$e_H(A\setminus R,\ D)=o\bigl(e_H(A,D)\bigr)$.
\end{definition}

\begin{definition} \label{def:aligned-rows}
Let $M\in\{0,1\}^{A\times D}$ be the $A$--$D$ incidence matrix and let $\bm{v}\in\mathbb R^D$ be a unit nonnegative top right singular vector of $M$
(so $\|M \bm{v}\|_2=\|M\|$).
For a vertex $a\in A$ with row vector $\bm{m}_a$, we write $\bm{x}_a:=\bm{m}_a/\|\bm{m}_a\|_2$ if $\bm{m}_a\neq 0$.
Given $\theta\in[0,1]$, we say that the $a$-th row is \emph{$\theta$-aligned} with $\bm{v}$ if $\langle \bm{x}_a,\bm{v}\rangle^2\ge 1-\theta$, and we write
\[
 R_\theta(M,\bm{v}) 
 :=
 \big\{a\in A: \bm{m}_a\neq 0\text{ and }\langle \bm{x}_a,\bm{v}\rangle^2\ge 1-\theta \big\}.
\]
\end{definition}

The next theorem shows that if the induced subgraph $H[C]$ is sparse, then the $A$--$D$ incidence matrix \(M\) is almost rank one, i.e.,
\[
        \|M\| =\bigl(1- o(1) \bigr) \|M\|_{\mathrm{F}}.
\]
Consequently, either $H$ already
contains many copies of $\Ktt$, or else the $A$--$D$ edges are essentially covered by a small set of vertices in $A$.

\begin{theorem} \label{thm:sparseC-rowcover}
Let $t\ge 2$ and $H$ be a graph with $m \to\infty$ edges. 
Assume that $H$ admits a partition $V(H)=A\sqcup C\sqcup D$ satisfying conditions \textup{(T1)}--\textup{(T3)},
and suppose that $e_H(A,C)=o(m)$ and  $|A|=o(\sqrt{m}\,)$. 
If in addition $\lambda (H)\ge \sqrt{m}$ and  $e(H[C])=o(m)$, 
then at least one of the following holds:
\begin{enumerate}
  \item[\rm(i)] We have $\texttt{\#}\Ktt(H)
     \ge 
    \left(\frac{1}{t!\,t^t}-o(1)\right) m^t 
    =\big( 1-o(1) \big)\texttt{\#}K_{t,t}(S_{t,m}).$

  \item[\rm(ii)]  There exist a subset $R\subseteq A$ with $1\le |R|\le t-1$, and a subset $B\subseteq D$ such that
  \[
    |B|=(1-o(1))\frac{m}{|R|},
    \qquad
    e_H(A\setminus R,\ B)=o(m),
    \qquad
    e_H(R,\ D\setminus B)=o(m),
  \]
  and $H[R,B]$ is complete bipartite (in particular $B$ is independent).
  Consequently, $R$ is an asymptotic $(t\!-\!1)$-row cover of the $A$--$D$ edge set in the sense of Definition~\ref{def:row-cover}.
\end{enumerate} 
\end{theorem}

\begin{proof} 
Let $H_{AD}$ be the bipartite graph with parts $A,D$ and edge set $E_H(A,D)$, and let
$M\in\{0,1\}^{A\times D}$ be its incidence matrix.
Then $\lambda(H_{AD})=\|M\|$ and $\|M\|_{\mathrm F}^2=e_H(A,D)$. 
By \textup{(T3)}, every edge of $H$ not in $H_{AD}$ lies in $H[C]$, between $A$ and $C$, or inside $A$.
Since $e(H[C])=o(m)$, $e_H(A,C)=o(m)$ and $|A| =o(\sqrt{m})$, it follows that 
\[
  e(H)-e_H(A,D)
  \ \le\
  e(H[C]) + e_H(A,C) + e(H[A])
  \ =\ o(m) + \binom{|A|}{2}
  \ =\ o(m).
\]
Denote $e:=e_H(A,D)$, so $e=(1-o(1))m$.
Since $\lambda(H\setminus H_{AD})\le \sqrt{2e(H\setminus H_{AD})}=o(\sqrt{m})$, and  the Rayleigh quotient gives 
\begin{equation}\label{eq:sparseC-triangle}
  \lambda(H) \le \lambda (H_{AD}) + \lambda (H\setminus H_{AD}) \le \|M\|+o(\sqrt{m}).
\end{equation} 
Substituting the assumption $\lambda(H)\ge \sqrt{m}$ into \eqref{eq:sparseC-triangle} gives 
$\|M\| \ge (1-o(1)) \sqrt{m}$. 
Trivially, we have $\|M\|\le \|M\|_{\mathrm F}=\sqrt e\le \sqrt{m}$. Thus, we conclude  $\|M\|=(1-o(1))\sqrt{m}$ and $ \|M\|^2 = (1-o(1)) m =(1-o(1))e$. This shows that $M$ has near rank one. 

In the following, we write $\rho:=\|M\|$ and set 
$\varepsilon:=1-\rho^2/e=o(1)$, and 
$ \theta:=\sqrt{\varepsilon} = o(1)$. 
Let $\bm{v} \in \mathbb{R}^D$ be a unit nonnegative top right singular vector of $M$ (that is, $\|M \bm{v}\| = \rho $). Let $R:=R_\theta(M, \bm{v})\subseteq A$
be the set of $\theta$-aligned rows  (see Definition~\ref{def:aligned-rows}).

Next, we prove that $R$ meets almost all $A$--$D$ edges. 
For each $a\in A$, let $\bm{m}_a$ be the $a$-th row vector of $M$. Note that $\|\bm{m}_a\|_2^2$ is equal to the degree $d_D(a):=|N_H(a)\cap D|$.
If $d_D(a)>0$, then $\bm{x}_a:=\bm{m}_a/\sqrt{d_D(a)}$ is a unit vector.
Since $\rho^2=\|M \bm{v} \|_2^2=\sum_{a\in A}\langle \bm{m}_a,\bm{v}\rangle^2$, we get 
\[
  e-\rho^2
  =
  \sum_{a\in A:\,d_D(a)>0} d_D (a) \cdot \left(1-\langle \bm{x}_a,\bm{v}\rangle^2\right).
\]
If $a\notin R$ and $d_D(a)>0$, then $1-\langle \bm{x}_a,\bm{v}\rangle^2>\theta$ by definition. Then 
\[
  e-\rho^2\ \ge\ \theta\sum_{a\in A\setminus R}d_D(a)\ =\ \theta \cdot e_H(A\setminus R,D).
\]
Since $e-\rho^2=\varepsilon e=\theta^2 e$, $\theta =o(1)$ and $e =(1-o(1))m$, 
we get
\begin{equation}\label{eq:sparseC-aligned-rowcover}
  e_H(A\setminus R,D)\ \le\ \theta e\ =\ o(m).
\end{equation}
In particular, we see that $R\neq\varnothing$ since $e_H(A,D)=(1-o(1))m$. 

In what follows, we prove that 
aligned rows of $R$ have large intersection. 
Fix distinct vertices $a,a'\in R$ with $d_D(a')\ge d_D(a)>0$. 
There exist a real number $\alpha$ and a vector $\bm{u}$ with $\bm{u}\perp \bm{v}$ such that 
\[
        \bm{x}_a=\alpha \bm{v}+ \bm{u}.
\]
Then $\alpha=\langle \bm{x}_a,\bm v\rangle\ge0$, and $\alpha^2\ge 1-\theta$, by the definition of \(R\), and hence
\[
        \|\bm{u}\|_2^2=1-\alpha^2\le \theta.
\]
Similarly, writing \(\bm{x}_{a'}=\alpha'\bm v+\bm u'\), we have $\alpha'\ge0,$ and $
\|\bm{u}'\|_2^2\le \theta.$
Using the Cauchy--Schwarz inequality, we get 
\begin{equation}
  \langle \bm{x}_a, \bm{x}_{a'} \rangle
  =\alpha\alpha'+\langle \bm{u}, \bm{u}'\rangle
  \ge (1-\theta)-\|\bm{u}\|_2\|\bm{u}'\|_2
  \ge 1-2\theta.
\end{equation}
Consequently, it follows that 
\begin{equation}\label{eq:sparseC-aligned-intersection}
  |N_D(a)\cap N_D(a') |
  =
  \langle \bm{m}_a,\bm{m}_{a'}\rangle
  =
  \sqrt{d_D(a)d_D(a')}\,\langle \bm{x}_a,\bm{x}_{a'}\rangle
  \ge (1-2\theta)\sqrt{d_D(a)d_D(a')}.
\end{equation}
Since the left-hand side is at most $d_D(a)$, we also have $d_D(a)\le d_D(a')\le d_D(a)/(1-2\theta)^2$. In other words, the degrees of any two vertices in $R$ are comparable. 

For any subset $S\subseteq R$ of size $s:=|S|$, let $a_\ast\in S$ be a vertex of minimum degree in $D$. Then for 
each $a\in S\setminus\{a_\ast\}$, we have 
$d_D(a)\ge d_D(a_\ast)$. 
By \eqref{eq:sparseC-aligned-intersection}, we get 
$|N_D(a_\ast) \cap N_D(a)| \ge (1-2\theta)d_D(a_\ast)$, which yields 
$|N_D(a_\ast) \setminus N_D(a)|\le 2\theta \cdot d_D(a_\ast)$.
Writing $N_D(S):=\bigcap_{a\in S}N_D(a)$, it follows that 
\begin{equation}\label{eq:sparseC-aligned-common-neighborhood}
  |N_D(S)|\ \ge\ (1-2(s-1)\theta) \cdot d_D(a_\ast).
\end{equation}

\smallskip
\noindent
{\bf  Case (i).} 
{\it Assume that $|R|\ge t$.}
Write $r:=|R|$ and $D_\ast:=\min_{a\in R}d_D(a)$. 
From \eqref{eq:sparseC-aligned-rowcover}, we have $\sum_{a\in R}d_D(a)=e_H(R,D)\ge (1-\theta)e$.
Also $\sum_{a\in R}d_D(a)\le e$, so $rD_\ast\le e$.
On the other hand, degree comparability gives $\max_{a\in R}d_D(a)\le D_\ast/(1-2\theta)^2$, hence
\[
  (1-\theta)e\ \le\ \sum_{a\in R}d_D(a)\ \le\ r\cdot \frac{D_\ast}{(1-2\theta)^2},
\]
which yields $D_\ast\ \ge\ (1-\theta)(1-2\theta)^2\,\frac{e}{r}.$
Since $r\le |A|=o(\sqrt{m})$, this implies 
\[
        D_\ast = (1-o(1)) \frac{e}{r}= \Omega (\sqrt{m}) \to\infty.
\]

Every $t$-subset $T\subseteq R$ has at least
$L:=\lfloor (1-2(t-1)\theta)D_\ast\rfloor$ common neighbours in $D$ by \eqref{eq:sparseC-aligned-common-neighborhood}. 
Choosing any \(t\) of these common neighbours gives a copy of \(\Ktt\) with one side \(T\) and the other side in \(D\). Hence
\[
  \texttt{\#}\Ktt(H)\ \ge\ \binom{r}{t}\binom{L}{t}.
\]
Note that $L=(1-o(1))e/r\to\infty$. So $\binom{L}{t}=\bigl(\frac{1}{t!}-o(1)\bigr)(e/r)^t$.
Since $r\ge t$,
\[
  \frac{\binom{r}{t}}{r^t}
  =
  \frac{1}{t!}\prod_{k=0}^{t-1}\left(1-\frac{k}{r}\right)
  \ge \frac{1}{t!}\prod_{k=0}^{t-1}\left(1-\frac{k}{t}\right)
  =
  \frac{1}{t^t}.
\]
Combining these bounds and $e=(1-o(1))m$ gives alternative~\rm(i).

\smallskip
\noindent
{\bf  Case (ii).} 
\emph{Assume that $|R|\le t-1$.}
Let $B:=N_D(R)=\bigcap_{a\in R}N_D(a)$.
Then $H[R,B]$ is complete bipartite and $B$ is an independent set. 
By \eqref{eq:sparseC-aligned-rowcover}, we have $e_H(A\setminus R,D)=o(m)$. Hence
\[
  e_H(R,D)\ =\ e_H(A,D)-e_H(A\setminus R,D)\ =\ (1-o(1))m.
\]
Recall that $D_\ast:=\min_{a\in R}d_D(a)$.
Degree comparability gives $\max_{a\in R}d_D(a)\le D_\ast/(1-2\theta)^2$, so
\[
  |R| D_\ast\ \le\ \sum_{a\in R}d_D(a)\ =\ e_H(R,D)\ \le\ |R| \cdot \frac{D_\ast}{(1-2\theta)^2},
\]
and therefore $D_\ast=(1-o(1)) \frac{m}{|R|}$, since $\theta=o(1)$.
Applying \eqref{eq:sparseC-aligned-common-neighborhood} with $S=R$ shows
\[
  |B|\ =\ (1-o(1))D_\ast\ =\ (1-o(1))\frac{m}{|R|}.
\]
Finally, since $e_H(R,B)=|R||B|=(1-o(1))m$ while $e_H(R,D)=(1-o(1))m$, we get
\[
        e_H(R,D\setminus B)=o(m).
\]
Also $e_H(A\setminus R,B)\le e_H(A\setminus R,D)=o(m)$.
This is alternative~\rm(ii).
\end{proof}

\subsubsection{Proof of Lemma \ref{lem:lemmaD-localized-sparse}}
\label{sec:4-5-2}

Note that $\lambda(H)>\lambda(S_{t-1,m}) \ge \sqrt{m}$ for all sufficiently large \(m\). 
Applying Theorem \ref{thm:sparseC-rowcover}, 
if part \textup{(i)} occurs, then
\[
\texttt{\#}\Ktt(H)
     \ge 
    \left(\frac{1}{t!\,t^t}-o(1)\right) m^t
    \ge \big( B_t-o(1) \big) m^t.
\]
Thus we are done. It remains to show that the small row-cover alternative \textup{(ii)} cannot occur.

Suppose on the contrary that Theorem \ref{thm:sparseC-rowcover} \textup{(ii)} occurs. 
An edge of $H$ not covered by the vertices of $R$ is called {\it bad}. We aim to bound the contribution of bad edges to the spectral radius. 
Let $E_b$ denote the set of all bad edges in $H$, and let $U$ be the set of vertices incident to any edge in $E_b$. 
Let $B\subseteq D$ be the set of common neighbors of $R$ from
Theorem~\ref{thm:sparseC-rowcover} \textup{(ii)}.
Since $H[R,B]$ is complete bipartite and $|B|=(1-o(1))m/|R|$, we have $e_H(R,B)\ =\ |R||B|\ =\ (1-o(1))m.$ Every edge of $E_b$ has no endpoint in $R$, hence $E_b \cap E_H(R,B)=\varnothing$ and therefore
\[
 |E_b|\ \le\ m-e_H(R,B)\ =\ o(m).
\]
In particular, we have $m-|E_b|\to\infty$.

If $|E_b|=0$, then every edge of $H$ meets $R$, so \(R\) is a vertex cover of size at most \(t-1\). Hence $H$ is $\Ktt$-free.
For sufficiently large $m$, applying Theorem \ref{thm-Ktt-free} implies $\lambda(H)\le \lambda(S_{t-1,m})$, which contradicts the assumption. 
Now we may assume $|E_b|\ge 1$.

\smallskip
In what follows, we show that the bad edges have small Perron weight and small contribution to the Rayleigh quotient of $H$. 
Note that $U\cap R=\varnothing$ and $|U|\le 2|E_b|$. Moreover, every edge leaving $U$ goes to $R$:
if $uv\in E_H(U, V(H)\setminus U)$, then $u\in U$ and $v\notin U$, forcing $uv \notin E_b$; since \(u\notin R\), we must have $v\in R$.
Hence
\[
  e_H\!\bigl(U,\,V(H)\setminus U\bigr)\ \le \ e_H(U,R)\ \le\ |R||U|\ \le\ 2|R| |E_b|\ \le\ 2(t-1)|E_b|.
\]
Denote $\lambda_{U}:=\lambda(H[U])$. Since \(E(H[U])=E_b\), the Frobenius-norm bound gives $\lambda_{U} \le  \sqrt{2|E_b|} = o(\sqrt{m}).$
The assumption $\lambda(H)>\lambda(S_{t-1,m})$ implies $\lambda(H)\ge \sqrt{m}$ for all sufficiently large $m$. Hence
\[
        (\lambda(H)-\lambda_{U})^2=(1-o(1))m>0.
\]
Recall that $\mu(U):=\sum_{u\in U} x_u^2$ denotes the Perron mass of $U$ with respect to $\bm{x}$. 
Denote $W:=V(H)\setminus U$. 
Applying Lemma~\ref{lem:spectral-vertex-cut}\textup{(c)} to the cut $U\,\sqcup\,W$, it follows that  
\[
  \mu(U)
  \ \le\
  \frac{e_H(U,W)}{(\lambda(H)-\lambda_{U})^2+e_H(U,W)}
  \ =\ O_t\!\left(\frac{|E_b|}{m}\right).
\]
We denote $w(E_b; \bm{x}) := \sum_{uv\in E_b} x_ux_v$. 
Since $E_b=E(H[U])$, the Rayleigh quotient gives  
\[
  2w(E_b; \bm{x})
  =\bm{x}_{U}^{\mathsf T}A(H[U]) \bm{x}_U
  \le
  \lambda_U \cdot \| \bm{x}_U\|_2^2
  =
  \lambda_U\cdot \mu(U),
\]
so $w(E_b; \bm{x})
        \le
        \frac12\lambda_U \mu(U)
        =
        O_t\!\left(\frac{|E_b|^{3/2}}{m}\right).$
Since $|E_b|=o(m)$, we have
\begin{equation}
    \label{eq-bad-edge-small}
    w(E_b; \bm{x}) = o \left(|E_b| \big/\sqrt{m} \right). 
\end{equation} 
Let $H_0 $ be the subgraph of $H$ obtained by deleting all edges of $E_b$.
Then $e(H_0)=m- |E_b|$. By definition of $E_b$, every edge of $H_0$ is incident to a vertex of $R$.
Thus, $R$ is a vertex cover of $H_0$ with $|R|\le t-1$, so $H_0$ is $\Ktt$-free.
Applying Theorem \ref{thm-Ktt-free} to $H_0$, since \(m-|E_b|\to\infty\), gives
$\lambda(H_0) \le \lambda(S_{t-1,m- |E_b|}).$
Therefore, by the Rayleigh quotient, we have 
\begin{equation}\label{eq:sparseC-rowcover-lamH-upper}
  \lambda(H)\ \le \ \lambda(H_0)+2w(E_b; \bm{x}) \ \le\ \lambda(S_{t-1,m-|E_b|}) + 2w(E_b;\bm{x}).
\end{equation}
On the other hand, applying Lemma~\ref{lem:Skm-increment} with $k=t-1$ and $d=|E_b|$ yields
\begin{equation} \label{eq-gap-large}
      \lambda(S_{t-1,m-|E_b|})
  \ \le\
  \lambda(S_{t-1,m})-\frac{|E_b|}{2(t-1)(\sqrt{m}+t-1)}, 
\end{equation} 
for all sufficiently large $m$.
Using \eqref{eq-bad-edge-small}, we get $2w(E_b; \bm{x})
        =
        o\!\left(\frac{|E_b|}{\sqrt{m}+t-1}\right).$
Combining \eqref{eq:sparseC-rowcover-lamH-upper} and \eqref{eq-gap-large}, it follows that $\lambda(H)<\lambda(S_{t-1,m})$, contradicting the assumption. Thus,  part \textup{(ii)} of Theorem~\ref{thm:sparseC-rowcover} cannot occur.  This completes the proof of Lemma~\ref{lem:lemmaD-localized-sparse}.

\subsection{Putting things together}

To prove Theorem \ref{thm:sharp-Ktt}, we also need a lemma of the authors \cite[Lemma 2.3]{LiLiuZhang2025EdgeSpectralESS}, which says that a fixed multiplicative gap above \(\sqrt m\) forces delocalization of the Perron eigenvector. 

\begin{lemma}[Li--Liu--Zhang \cite{LiLiuZhang2025EdgeSpectralESS}] \label{lem:max} 
Let $H$ be a graph with $m$ edges and unit Perron eigenvector \(\bm{x}=(x_v)_{v\in V(H)}\).  
 If $\lambda^2 (H) \ge (1+\delta)m$ where $0<\delta \le 1/3$, then 
$\|\bm{x}\|_{\infty} m^{1/4} < \delta^{-4}$.  
\end{lemma}

We include a proof for the convenience of readers. 

\begin{proof} 
    We denote $\lambda = \lambda(H)$ and 
    $x_i=\max\{x_v: v\in V(H)\}$. 
   We actually show that $x_i m^{1/4} < \delta^{-2}$. 
    Suppose on the contrary that $x_i \ge \delta^{-2}m^{-1/4}$. Let $L=\{u\in V(H): x_u > \delta^{-1/2}m^{-1/4}\}.$
     Then 
    \begin{equation} \label{eq-square}
        \lambda^2x_i
        =
        \sum_{j\in N(i)} \sum_{k\in N(j)} x_k 
        \le
        \sum_{jk\in E(H)} (x_j +x_k) 
        =
        \sum_{jk\in E(H[L])} (x_j+x_k) 
        +
        \sum_{jk\in E(H)\setminus E(H[L])} (x_j+x_k).
    \end{equation}
   Since $\sum_{v\in V(H)}x_v^2=1$, we have 
    $|L|\le \delta m^{1/2}$. So the first sum on the right-hand side of \eqref{eq-square} is at most $2x_i\binom{|L|}{2}\le \delta^2 mx_i.$
    For each $jk\in E(H)\setminus E(H[L])$, 
    we have $\min\{x_j,x_k\}\le \delta^{-1/2}m^{-1/4} \le \delta^{3/2}x_i$. 
    Then $x_j +x_k \le x_i + \delta^{3/2}x_i= 
        (1+ \delta^{3/2})x_i.$
    Therefore, we obtain from \eqref{eq-square} that $(1+\delta)m x_i
        \le
        \lambda^2 x_i
        \le 
        \delta^2 mx_i + (1+ \delta^{3/2})x_im.$
    However, for any $0< \delta \le 1/3$, we have 
    $1+ \delta > \delta^2 + (1+ \delta^{3/2})$, which is a contradiction. 
\end{proof}

To start with the proof of Theorem \ref{thm:sharp-Ktt}, we explain how the key lemmas in Section \ref{sec-four-key-lemmas} combine.
Assume $t\ge 3$ throughout (the case $t=2$ is covered by Theorem \ref{thm:counting-C4}) and write
$B_t:=2^{-(t-1)^2}/(t!)^2$.
Lemma~\ref{lem:lemmaA-pruning} prunes $G$ to a subgraph $H$ satisfying the $\eta$-heavy-edge condition
(Definition~\ref{def:heavy-edge}) while preserving the split gap; if pruning deletes a positive fraction of edges, it also yields a fixed
multiplicative gap $\lambda(H)\ge (1+\delta)\sqrt{e(H)}$.
Lemmas~\ref{lem:lemmaB-delocalized}--\ref{lem:lemmaD-localized-sparse} then count many copies of $\Ktt$ in $H$, splitting according to the
localization parameter $g =\|\bm{x}\|_\infty e(H)^{1/4}$ of the Perron eigenvector.

\begin{proof}[Proof of Theorem \ref{thm:sharp-Ktt}]
Fix \(\varepsilon>0\). Suppose for contradiction that there exists a sequence of graphs \(\{G_n\}_{n\ge1}\) with $e(G_n)=m_n\to\infty,$ $\lambda(G_n)>\lambda(S_{t-1,m_n}),$
and $\texttt{\#}\Ktt(G_n)\le (B_t-\varepsilon)m_n^t.$
Let \(\eta:=1/(16t)\).
Applying Lemma~\ref{lem:lemmaA-pruning} to each \(G_n\), we obtain a subgraph $H_n\subseteq G_n$ with $m_n':=e(H_n)$,
$\alpha_n:=m_n'/m_n$, and unit Perron eigenvector $\bm{x}^{(n)}$ satisfying the heavy-edge condition and
$\lambda(H_n)>\lambda(S_{t-1,m_n'})$. 
Passing to a subsequence, we may assume $\alpha_n\to\alpha\in[0,1]$ and set
\[
        g_n:=\|\bm{x}^{(n)}\|_\infty(m_n')^{1/4}.
\]
Lemma~\ref{lem:lemmaA-pruning} gives
\[
        \frac{\lambda(H_n)}{\sqrt{m_n'}}
        \ge
        1+(1-4\eta)(\alpha_n^{-1/2}-1).
\]
Since $\alpha_n := \frac{m_n'}{m_n} \ge c_{\eta} $, where $c_{\eta}:=\left(\frac{1-4\eta}{\sqrt{2}-4\eta}\right)^2\in(0,1)$
by Claim \ref{cl:H-linearly-size}, the numbers $\alpha_n$ are bounded away from $0$; in particular
$\alpha\in(0,1]$.
Passing to a further subsequence, assume either $g_n=O(1)$ or $g_n\to\infty$.

\smallskip

\noindent{\bf Case 1. $\alpha=1$.}
Then $(m_n')^t=(1+o(1))m_n^t$.
If $g_n=O(1)$, then Lemma~\ref{lem:lemmaB-delocalized} gives
$\texttt{\#}\Ktt(H_n)
        \ge
        \bigl(B_t-o(1)\bigr)(m_n')^t,$
where we used $\lambda(H_n)>\lambda(S_{t-1,m_n'})\ge \sqrt{m_n'}$ for all sufficiently large \(n\).
Otherwise, if $g_n\to\infty$, then the structural result of Theorem~\ref{thm:localized-structure} yields an ACD partition
\[
        V(H_n)=A_n\sqcup C_n\sqcup D_n.
\]
If $e(H_n[C_n])=(\beta+o(1))m_n'$ for some constant $\beta\in(0,1]$ along a subsequence, then
Lemma~\ref{lem:lemmaC-localized-nonsparse}  applies; otherwise, if $e(H_n[C_n])=o(m_n')$, then 
Lemma~\ref{lem:lemmaD-localized-sparse} applies.
In all subcases, we get
\[
    \texttt{\#}\Ktt(G_n)
        \ge
        \texttt{\#}\Ktt(H_n)
        \ge
        \bigl(B_t-o(1)\bigr)m_n^t,
\]
a contradiction.

\smallskip

\noindent{\bf Case 2. $\alpha<1$.}
Discarding finitely many terms, we may assume $\alpha_n\le 1-\varepsilon_0$ for some fixed $\varepsilon_0>0$.
Denote $\delta_n:=(1-4\eta)(\alpha_n^{-1/2}-1).$
Then Lemma~\ref{lem:lemmaA-pruning} gives
$\lambda(H_n)\ge (1+\delta_n)\sqrt{m_n'}.$
Since $\alpha_n\to\alpha\in(0,1)$, we have $\delta_n\to\delta :=(1-4\eta)(\alpha^{-1/2} -1)\in(0,\infty).$
Choose a fixed \(\zeta\in(0,1/3]\) such that $(1+\delta_n)^2\ge 1+\zeta$
for all sufficiently large \(n\). Lemma~\ref{lem:max} then gives \(g_n=O(1)\).
Then we can apply Lemma~\ref{lem:lemmaB-delocalized} and get that 
\[
 \texttt{\#}\Ktt(H_n)
 \ge
 \bigl(B_t(1+\delta_n)^{2t(t-1)}-o(1)\bigr)\alpha_n^t m_n^t.
\] 

We now show $\alpha_n^t(1+\delta_n)^{2t(t-1)}\ge 1.$
This is equivalent to $\big(1+(1-4\eta)(\alpha_n^{-1/2} -1)\big)^{t-1} \ge \alpha_n^{-1/2}.$
Denote $\beta_n:=\alpha_n^{-1/2}> 1$ and $c:=1-4\eta$. 
Then we need to prove that 
\[
        \big(1+c(\beta_n-1) \big)^{t-1} \ge \beta_n.
\]
Since \(\eta=1/(16t)\) and $t\ge 3$, we have $c> \frac{1}{t-1}.$
Let
\[
        f(x):=(1+cx-c )^{t-1}-x,
        \qquad x\in (1,\infty).
\]
Then        $f''(x) =(t-1)(t-2)c^2(1+cx-c)^{t-3} > 0.$
Thus
$f'(x)> f'(1)=(t-1)c-1>0$
for all \(x>1\), and hence
$f(x)>f(1)=0$
for all \(x>1\). Setting \(x=\beta_n>1\) gives the desired inequality.
Thus,
\[
     \texttt{\#}\Ktt(G_n)
        \ge
        \texttt{\#}K_{t,t}(H_n)
        \ge
        \bigl(B_t-o(1)\bigr)m_n^t,
\]
again contradicting \(\texttt{\#}\Ktt(G_n)\le (B_t-\varepsilon)m_n^t\).
\end{proof}

\section{Proof of Theorem \ref{thm:sharp-C2t}}

\label{sec:even-cycles}

We first record the basic spectral counting input for even cycles.

\begin{lemma}
\label{lem:spectral-count-C2t}
Let \(G\) be an \(n\)-vertex graph with spectral radius \(\lambda\).
Then, for every fixed \(t\ge2\),
\[
        \texttt{\#}C_{2t}(G)
        \ge
        \frac{1}{4t}\lambda^{2t}-O_t(n^{2t-1}).
\]
\end{lemma}

\begin{proof}
The number \(\homc(C_{2t},G)\) is the number of closed walks of length
\(2t\) in \(G\).  Hence
\[
        \homc(C_{2t},G)
        =
        \operatorname{tr}A(G)^{2t}
        =
        \sum_i \lambda_i(G)^{2t}
        \ge
        \lambda(G)^{2t}.
\]
The number of non-injective maps from \(V(C_{2t})\) to \(V(G)\) is
\(O_t(n^{2t-1})\).  Dividing the number of injective homomorphisms by
\(|\Aut(C_{2t})|=4t\) gives the result.
\end{proof}

The sharp supersaturation behaviour for even cycles is different from
that for complete bipartite graphs.  For comparison, the random graphs used in
Example~\ref{exam:sharpness-Bt} give the constant \(1/(4t)\) for
\(C_{2t}\), whereas the split graphs \(S_{t,m}\) give the smaller
constant $\frac{(t-1)!}{2t^t}$ for every \(t\ge3\).  Thus the sharp construction for \(C_{2t}\) is
split graphs rather than balanced dense graphs. Indeed, let \(G\sim G(n,m)\) with \(n=2\sqrt m+O_t(1)\).  Then
\[
 \mathbb E[\texttt{\#}C_{2t}(G)]
 =
 {n\choose 2t}\frac{(2t)!}{4t}
 \cdot
 \frac{\binom{N-2t}{m-2t}}{\binom{N}{m}}
 =
 \left(\frac{1}{4t}+O_t(m^{-1/2})\right)m^t,
\]
where \(N=\binom n2\).  On the other hand, if
\(m-\binom t2=tq+r\) with \(0\le r\le t-1\), then the leading
contribution to \(\texttt{\#}C_{2t}(S_{t,m})\) comes from choosing all
\(t\) clique vertices and \(t\) vertices from the independent side.
Thus
\[
        \texttt{\#}C_{2t}(S_{t,m})
        =
        \left(1+o(1)\right)
        {q\choose t}\frac{t!(t-1)!}{2}
        =
        \left(\frac{(t-1)!}{2t^t}+o(1)\right)m^t .
\]

We next state the cycle analogues of the counting lemmas used for
\(K_{t,t}\).  Their proofs are the same as the corresponding arguments
in Section~\ref{sec-four-key-lemmas}, with Lemma~\ref{lem:spectral-count-C2t}
replacing Lemma~\ref{lem:spectral-count-Ktt}.

\begin{lemma}[Delocalized cycle counting]
\label{lem:cycle-delocalized}
Fix \(t\ge2\) and \(\eta\in(0,\frac14)\).  Let \(H\) be an \(m\)-edge
graph with unit Perron eigenvector \(\bm{x}\), and set $g:=\|\bm{x}\|_\infty m^{1/4}.$
If \(H\) satisfies the \(\eta\)-heavy-edge condition and \(g=O(1)\), then $\texttt{\#}C_{2t}(H)
        \ge
        \left(\frac{1}{4t}-o(1)\right)\lambda(H)^{2t}.$
\end{lemma}

\begin{proof}
Delete isolated vertices, if any.  This does not change the number of
edges, the spectral radius, the Perron vector on the non-isolated part,
or the number of cycles.  The proof of Lemma~\ref{lem:lemmaB-delocalized}
then gives \(|V(H)|=O(\sqrt m)\).  By Lemma~\ref{lem:spectral-count-C2t},
$\texttt{\#}C_{2t}(H)
        \ge
        \frac{1}{4t}\lambda(H)^{2t}-O_t(|V(H)|^{2t-1}).$
The heavy-edge condition gives
$\lambda(H)
        =
        2\sum_{uv\in E(H)}x_ux_v
        \ge
        2\eta\sqrt m.$
Thus \(\lambda(H)^{2t}=\Omega_{t,\eta}(m^t)\), while $|V(H)|^{2t-1}=O(m^{t-1/2})=o(m^t).$
The error term is therefore \(o(\lambda(H)^{2t})\).
\end{proof}

\begin{lemma}[Localized dense core:  cycles counting]
\label{lem:cycle-localized-nonsparse}
Let \(H\) be an \(m\)-edge graph satisfying the \(\eta\)-heavy-edge
condition for some fixed \(\eta\in(0,\frac14)\), and suppose
\(g=\|\bm{x}\|_\infty m^{1/4}\to\infty\).  Assume also that
\(\lambda(H)\ge\sqrt m\).  Let
$V(H)=A\sqcup C\sqcup D$
be the ACD partition obtained from Theorem~\ref{thm:localized-structure}.
If $e(H[C])=(\alpha+o(1))m$
for some constant \(\alpha\in(0,1]\), then $\texttt{\#}C_{2t}(H)
        \ge
        \left(\frac{1}{4t}-o(1)\right)\lambda(H)^{2t}.$
In particular, $\texttt{\#}C_{2t}(H)
        \ge
        \left(\frac{1}{4t}-o(1)\right)m^t.$
\end{lemma}

\begin{proof}
The proof of Lemma~\ref{lem:lemmaC-localized-nonsparse} gives $\lambda(H[C])=\lambda(H)-o(\sqrt m).$
Since \(\lambda(H)\ge\sqrt m\), this implies $\lambda(H[C])=(1-o(1))\lambda(H).$
Also Theorem~\ref{thm:localized-structure} gives $|C|\le \sqrt m\,m^{o(1)}.$
Applying Lemma~\ref{lem:spectral-count-C2t} to \(H[C]\), we get
\[
        \texttt{\#}C_{2t}(H)\ge
        \texttt{\#}C_{2t}(H[C])\ge
        \frac{1}{4t}\lambda(H[C])^{2t}
        -
        O_t(|C|^{2t-1})     =
        \left(\frac{1}{4t}-o(1)\right)\lambda(H)^{2t},
\]
as $|C|^{2t-1}
        \le
        \left(\sqrt m\,m^{o(1)}\right)^{2t-1}
        =
        o(m^t)
        =
        o(\lambda(H)^{2t}).$
\end{proof}

\begin{lemma}
\label{lem:cycle-sparseC-rowcover}
Let \(H\) be a graph with \(m\to\infty\) edges.  Assume that \(H\)
admits a partition \(V(H)=A\sqcup C\sqcup D\) satisfying
\textup{(T1)}--\textup{(T3)}, and suppose moreover that $e_H(A,C)=o(m),$ $
        |A|=o(\sqrt m),$ and $        e(H[C])=o(m).$
If \(\lambda(H)\ge\sqrt m\), then at least one of the following holds.
\begin{enumerate}
    \item[\rm(i)] $\texttt{\#}C_{2t}(H)
        \ge
        \left(\frac{(t-1)!}{2t^t}-o(1)\right)m^t.$

    \item[\rm(ii)]
    There exist \(R\subseteq A\) with \(1\le |R|\le t-1\), and
    \(B\subseteq D\), such that
    \[
        |B|=(1-o(1))\frac{m}{|R|},
        \qquad
        e_H(A\setminus R,B)=o(m),
        \qquad
        e_H(R,D\setminus B)=o(m),
    \]
    and \(H[R,B]\) is complete bipartite.
\end{enumerate}
\end{lemma}

\begin{proof}
The proof is identical to the proof of
Theorem~\ref{thm:sparseC-rowcover} until the final counting step.  We
briefly record the modification.  Let \(M\) be the \(A\)--\(D\)
incidence matrix.  As in Theorem~\ref{thm:sparseC-rowcover}, one obtains
a set \(R\subseteq A\) of aligned rows such that
\[
        e_H(A\setminus R,D)=o(m).
\]
If \(|R|\le t-1\), then the same argument gives alternative \textup{(ii)}.

Assume now that \(|R|=r\ge t\).  Let
\[
        D_\ast:=\min_{a\in R} d_D(a).
\]
As in the proof of Theorem~\ref{thm:sparseC-rowcover}, the degrees
\(d_D(a)\), \(a\in R\), are pairwise comparable, and
\[
        D_\ast=(1-o(1))\frac{e}{r},
        \qquad
        e:=e_H(A,D)=(1-o(1))m.
\]
Moreover every \(t\)-subset \(T\subseteq R\) has at least $L=(1-o(1))D_\ast$
common neighbours in \(D\).  For each such \(T\), choosing \(t\) common
neighbours in \(D\) gives a copy of \(K_{t,t}\), which contains
$        \frac{t!(t-1)!}{2}$
copies of \(C_{2t}\).  Hence
\[
        \texttt{\#}C_{2t}(H)\ge
        \binom rt\binom Lt\frac{t!(t-1)!}{2}       \ge
        \left(\frac{(t-1)!}{2t^t}-o(1)\right)e^t  =
        \left(\frac{(t-1)!}{2t^t}-o(1)\right)m^t.
\]
This gives alternative \textup{(i)}.
\end{proof}

\begin{lemma}[Localized sparse core: cycle counting]
\label{lem:cycle-localized-sparse}
Let \(H\) be an \(m\)-edge graph admitting an ACD partition
\(V(H)=A\sqcup C\sqcup D\) satisfying $e_H(A,C)=o(m),$ $       |A|=o(\sqrt m),$ and $e(H[C])=o(m).$
If $\lambda(H)>\lambda(S_{t-1,m}),$
then
\[
        \texttt{\#}C_{2t}(H)
        \ge
        \left(\frac{(t-1)!}{2t^t}-o(1)\right)m^t .
\]
\end{lemma}

\begin{proof}
For all sufficiently large \(m\), the assumption gives $\lambda(H)>\lambda(S_{t-1,m})\ge\sqrt m.$
Apply Lemma~\ref{lem:cycle-sparseC-rowcover}.  If alternative
\textup{(i)} occurs, we are done.

Suppose alternative \textup{(ii)} occurs.  The same bad-edge argument as
in the proof of Lemma~\ref{lem:lemmaD-localized-sparse} gives a
contradiction.  Indeed, call an edge of \(H\) \emph{bad} if it is not
incident to \(R\), and let \(E_b\) be the set of bad edges.  Then
$|E_b|=o(m).$
If \(E_b=\varnothing\), then \(R\) is a vertex cover of size at most
\(t-1\), so \(H\) is \(K_{t,t}\)-free.  Theorem~\ref{thm-Ktt-free}
then gives $\lambda(H)\le\lambda(S_{t-1,m}),$
contrary to the assumption.

Thus assume \(|E_b|\ge1\).  Let \(U\) be the set of vertices incident
to an edge of \(E_b\).  As in Lemma~\ref{lem:lemmaD-localized-sparse},
\[
        e_H(U,V(H)\setminus U)=O_t(|E_b|),
        \qquad
        \lambda(H[U])\le \sqrt{2|E_b|}=o(\sqrt m),
\]
and Lemma~\ref{lem:spectral-vertex-cut}\textup{(c)} implies that the
Perron mass of \(U\) is
$        \mu(U)=O_t(|E_b|/m).$
Therefore, for a unit Perron eigenvector \(\bm{x}\) of \(H\),
\[
        \sum_{uv\in E_b}x_ux_v
        =
        o\left(\frac{|E_b|}{\sqrt m}\right).
\]
Let \(H_0:=H-E_b\).  Then every edge of \(H_0\) is incident to \(R\), so
\(H_0\) is \(K_{t,t}\)-free.  Hence
$\lambda(H_0)\le \lambda(S_{t-1,m-|E_b|})$
by Theorem~\ref{thm-Ktt-free}.  By the Rayleigh quotient,
$\lambda(H)
        \le
        \lambda(H_0)+2\sum_{uv\in E_b}x_ux_v.$
Using Lemma~\ref{lem:Skm-increment},
\[
        \lambda(S_{t-1,m-|E_b|})
        \le
        \lambda(S_{t-1,m})
        -
        \frac{|E_b|}{2(t-1)(\sqrt m+t-1)}.
\]
The contribution of the bad edges is
\(o(|E_b|/(\sqrt m+t-1))\), so the last three displays imply $\lambda(H)<\lambda(S_{t-1,m}),$
a contradiction.  Hence alternative \textup{(ii)} cannot occur.
\end{proof}

We now prove Theorem~\ref{thm:sharp-C2t}.

\begin{proof}[Proof of Theorem~\ref{thm:sharp-C2t}]
For \(t=2\), the result is exactly Theorem~\ref{thm:counting-C4}, since
\(C_4=K_{2,2}\) and \(S_{1,m}\) is the star with \(m\) edges.  Hence
assume \(t\ge3\).  Set $c_t:=\frac{(t-1)!}{2t^t}.$
Fix \(\varepsilon>0\).  Suppose, for contradiction, that there is a
sequence of graphs \(G_n\) with
$e(G_n)=m_n\to\infty,$ $\lambda(G_n)>\lambda(S_{t-1,m_n}),$ 
and $\texttt{\#}C_{2t}(G_n)\le (c_t-\varepsilon)m_n^t.$
Let \(\eta:=1/(16t)\).  Applying Lemma~\ref{lem:lemmaA-pruning} to each
\(G_n\), we obtain a subgraph \(H_n\subseteq G_n\) with
$m_n':=e(H_n),$ $
        \alpha_n:=\frac{m_n'}{m_n},$
and a unit Perron eigenvector \(\bm{x}^{(n)}\), such that \(H_n\)
satisfies the \(\eta\)-heavy-edge condition and
$\lambda(H_n)>\lambda(S_{t-1,m_n'}).$
Moreover,
\[
        \frac{\lambda(H_n)}{\sqrt{m_n'}}
        \ge
        1+(1-4\eta)(\alpha_n^{-1/2}-1).
\]
By Claim~\ref{cl:H-linearly-size}, the numbers \(\alpha_n\) are bounded
away from \(0\).  Passing to a subsequence, assume
$\alpha_n\to\alpha\in(0,1],$ $       g_n:=\|\bm{x}^{(n)}\|_\infty(m_n')^{1/4}$
satisfies either \(g_n=O(1)\) or \(g_n\to\infty\).

\smallskip
\noindent\textbf{Case 1: \(\alpha=1\).}
Then \((m_n')^t=(1+o(1))m_n^t\).

If \(g_n=O(1)\), Lemma~\ref{lem:cycle-delocalized} gives
$\texttt{\#}C_{2t}(H_n)
        \ge
        \left(\frac{1}{4t}-o(1)\right)\lambda(H_n)^{2t}.$
Since \(\lambda(H_n)>\lambda(S_{t-1,m_n'})\ge\sqrt{m_n'}\), we get
$\texttt{\#}C_{2t}(H_n)
        \ge
        \left(\frac{1}{4t}-o(1)\right)(m_n')^t.$
As
$        \frac{1}{4t}\ge \frac{(t-1)!}{2t^t}=c_t$
for all \(t\ge2\), this gives $
        \texttt{\#}C_{2t}(G_n)
        \ge
        (c_t-o(1))m_n^t,$
contradicting the assumption.

Now suppose \(g_n\to\infty\).  By Theorem~\ref{thm:localized-structure},
\(H_n\) admits an ACD partition
$V(H_n)=A_n\sqcup C_n\sqcup D_n.$
Passing to a further subsequence, write
$        e(H_n[C_n])=(\beta+o(1))m_n'$
for some \(\beta\in[0,1]\).  If \(\beta>0\), then
Lemma~\ref{lem:cycle-localized-nonsparse} gives
        $\texttt{\#}C_{2t}(H_n)
        \ge
        \left(\frac{1}{4t}-o(1)\right)(m_n')^t
        \ge
        (c_t-o(1))m_n^t.$
If \(\beta=0\), then Lemma~\ref{lem:cycle-localized-sparse} gives
$\texttt{\#}C_{2t}(H_n)
        \ge
        (c_t-o(1))(m_n')^t
        =
        (c_t-o(1))m_n^t.
$
Again this contradicts the assumption.

\smallskip
\noindent\textbf{Case 2: \(\alpha<1\).}
Discarding finitely many terms, assume
$\alpha_n\le 1-\varepsilon_0
$
for some fixed \(\varepsilon_0>0\).  Put
$        \delta_n:=(1-4\eta)(\alpha_n^{-1/2}-1).
$
Then
$        \lambda(H_n)\ge (1+\delta_n)\sqrt{m_n'}.$
Since \(\alpha_n\to\alpha\in(0,1)\), we have
$        \delta_n\to\delta:=(1-4\eta)(\alpha^{-1/2}-1)>0.
$
Choose \(\zeta\in(0,1/3]\) such that
$        (1+\delta_n)^2\ge 1+\zeta
$
for all sufficiently large \(n\).  Lemma~\ref{lem:max} gives
\(g_n=O(1)\).  Therefore Lemma~\ref{lem:cycle-delocalized} yields
\[
        \texttt{\#}C_{2t}(H_n)
        \ge
        \left(\frac{1}{4t}-o(1)\right)\lambda(H_n)^{2t} \ge
        \left(\frac{1}{4t}-o(1)\right)
        (1+\delta_n)^{2t}\alpha_n^t m_n^t.
\]

It remains to compare the constant with \(c_t\).  Let
$\beta_n:=\alpha_n^{-1/2}>1,$ and $
        c:=1-4\eta=1-\frac{1}{4t}.$
Then
$        (1+\delta_n)^{2t}\alpha_n^t
        =
        \left(\frac{1+c(\beta_n-1)}{\beta_n}\right)^{2t}.$
Since \(\beta_n>1\),
$        \frac{1+c(\beta_n-1)}{\beta_n}>c.
$
Hence
$        \frac{1}{4t}(1+\delta_n)^{2t}\alpha_n^t
        >
        \frac{1}{4t}\left(1-\frac{1}{4t}\right)^{2t}.$
By Bernoulli's inequality,
$        \left(1-\frac{1}{4t}\right)^{2t}\ge \frac12.
$
Thus the limiting coefficient is greater than \(1/(8t)\).  For
\(t\ge3\),
$        \frac{1}{8t}>\frac{(t-1)!}{2t^t}=c_t,
$
because \(4(t-1)!<t^{t-1}\).  Therefore
\[
        \texttt{\#}C_{2t}(G_n)
        \ge
        \texttt{\#}C_{2t}(H_n)
        \ge
        (c_t+\Omega_t(1)-o(1))m_n^t,
\]
contradicting
$        \texttt{\#}C_{2t}(G_n)\le (c_t-\varepsilon)m_n^t.
$
This completes the proof.
\end{proof}

\section{Concluding remarks}
\label{sec:Concluding}

We conclude with a few directions to explore.  The main phenomenon is that the same edge-spectral threshold
can have different sharp extremal mechanisms: for \(K_{t,t}\), the sharp
constant is attained by balanced random graphs on \((2+o(1))\sqrt m\)
vertices, while for \(C_{2t}\), it is attained by split graphs.

\smallskip
\noindent\textbf{Second-order \(\bm{C_4}\) supersaturation.}
Let \(f_4(m)\) be the minimum number of copies of \(C_4\) in an
\(m\)-edge graph \(G\) satisfying \(\lambda(G)>\sqrt m\).  Theorem
\ref{thm:counting-C4} gives $f_4(m)=\left(\frac18-o(1)\right)m^2.$
Two different near-extremal constructions, one clique-based and one
split-based, suggest the same second-order term (see~\cite{LiLiuZhang2025NosalBooksC4}).  It would be interesting
to determine whether
\[
        f_4(m)
        =
        \frac18m^2-\frac14m^{3/2}+o(m^{3/2}).
\]

\smallskip
\noindent\textbf{Odd cycles above split thresholds.}
For fixed \(t\ge2\), it is known that the split graph \(S_{t,m}\) is the
edge-spectral extremal construction for forbidding \(C_{2t+1}\).  The
natural supersaturation problem is to determine the sharp asymptotic
constant
\[
        \inf_{\lambda(G)>\lambda(S_{t,m})}
        \frac{\texttt{\#}C_{2t+1}(G)}{m^t}
\]
as \(m\to\infty\).  Unlike the even-cycle case, the extremal mechanism is
not clear from the present methods.

\smallskip
\noindent\textbf{Which Sidorenko graphs have split extremizers?}
The equivalence theorem shows that every Sidorenko graph \(v(H)\le e(H)\) admits a
spectral Sidorenko inequality. 
We developed a unified approach to establishing edge-spectral supersaturation for every Sidorenko graph $H$ when the split graph is an extremal $H$-free graph. 
A broader problem is to determine, for
bipartite Sidorenko graphs \(H\), the sharp edge-spectral supersaturation
constant above the relevant threshold and to classify the
asymptotic extremizers. The results of this paper show that both
balanced random graphs and split graphs can be sharp, even among very
classical Sidorenko graphs. Are there other asymptotic extremizers?

{ 
\smallskip
\noindent\textbf{Acknowledgment.} 
The authors would like to thank Ting-Wei Chao for helpful comments 
on an earlier draft.   
The authors introduced and developed the unified framework for edge-spectral supersaturation, and wrote and verified all proofs presented in this paper. During an early exploratory stage, the second and fourth authors additionally used language-model-based tools to brainstorm candidate proof strategies. Any suggestions arising from these tools served only as informal inspiration. 
}

\frenchspacing


\begin{thebibliography}{10}

\bibitem{BN2007jctb}
B.~Bollob{\'a}s and V.~Nikiforov.
\newblock Cliques and the spectral radius.
\newblock {\em J. Combin. Theory Ser. B}, 97:859--865, 2007.

\bibitem{Cerda2010linear}
J.~Cerd\`{a}.
\newblock {\em Linear Functional Analysis}.
\newblock Number 116 in GSM. American Mathematical Society, Providence, RI;
  Real Sociedad Matem\'{a}tica Espa\~nola, Madrid, 2010.

\bibitem{ChaoYu2026JLMS}
T.-W. Chao and H.-H.~H. Yu.
\newblock When entropy meets {Turán}: {New} proofs and hypergraph {Turán}
  results.
\newblock {\em J. Lond. Math. Soc.}, 113(3):Paper No. e70473, 2026.

\bibitem{CLLLN2026}
H.~Chen, J.~Li, Y.~Li, L.~Liu, and B.~Ning.
\newblock An exponentially small gap of the {Perron} vector on independent
  sets, 2026.
\newblock arXiv:2604.24077.

\bibitem{Chung1997SpectralGraphTheory}
F.~R.~K. Chung.
\newblock {\em Spectral Graph Theory}, volume~92 of {\em CBMS Regional
  Conference Series in Mathematics}.
\newblock American Mathematical Society, Providence, RI, 1997.

\bibitem{ConlonFoxSudakov2010}
D.~Conlon, J.~Fox, and B.~Sudakov.
\newblock An approximate version of {Sidorenko}'s conjecture.
\newblock {\em Geom. Funct. Anal.}, 20(6):1354--1366, 2010.

\bibitem{ConlonKimLeeLee2018}
D.~Conlon, J.~H. Kim, C.~Lee, and J.~Lee.
\newblock Some advances on {S}idorenko's conjecture.
\newblock {\em J. Lond. Math. Soc. (2)}, 98(3):593--608, 2018.

\bibitem{ConlonLee2017}
D.~Conlon and J.~Lee.
\newblock Finite reflection groups and graph norms.
\newblock {\em Adv. Math.}, 315:130--165, 2017.

\bibitem{CoreglianoRazborov2021Biregularity}
L.~N. Coregliano and A.~A. Razborov.
\newblock Biregularity in {Sidorenko}'s conjecture.
\newblock {\em arXiv preprint arXiv:2108.06599}, 2021.

\bibitem{Erd1962a}
P.~Erd\H{o}s.
\newblock On a theorem of {Rademacher--Tur\'{a}n}.
\newblock {\em Illinois J. Math.}, 6:122--127, 1962.

\bibitem{Erdos1964}
P.~Erdős.
\newblock On the number of triangles contained in certain graphs.
\newblock {\em Canadian Mathematical Bulletin}, 7(1):53--56, 1964.

\bibitem{FangLinZhai2026}
L.~Fang, H.~Lin, and M.~Zhai.
\newblock Counting color-critical subgraphs under {Nikiforov's} condition,
  2026.
\newblock arXiv:2603.14964.

\bibitem{FoxLovasz2017Adv}
J.~Fox and L.~M. Lov{\'a}sz.
\newblock A tight bound for green's arithmetic triangle removal lemma in vector
  spaces.
\newblock {\em Adv. Math.}, 321:287--297, 2017.

\bibitem{Grafakos2014ClassicalFourierAnalysis}
L.~Grafakos.
\newblock {\em Classical Fourier Analysis}.
\newblock Graduate Texts in Mathematics. Springer, New York, 3rd edition, 2014.

\bibitem{Hatami2010}
H.~Hatami.
\newblock Graph norms and {Sidorenko}'s conjecture.
\newblock {\em Israel J. Math.}, 175:125--150, 2010.

\bibitem{HornJohnson2012MatrixAnalysis}
R.~A. Horn and C.~R. Johnson.
\newblock {\em Matrix Analysis}.
\newblock Cambridge University Press, 2 edition, Oct. 2012.

\bibitem{KimLeeLee2016}
J.~H. Kim, C.~Lee, and J.~Lee.
\newblock Two approaches to {Sidorenko}'s conjecture.
\newblock {\em Trans. Amer. Math. Soc.}, 368(7):5057--5074, 2016.

\bibitem{LZZ2024}
S.~Li, S.~Zhao, and L.~Zou.
\newblock Spectral extrema of graphs with fixed size: {Forbidden} a fan graph,
  a friendship graph or a theta graph.
\newblock {\em J. Graph Theory}, 110(4):483–495, 2025.

\bibitem{LZS2024}
X.~Li, M.~Zhai, and J.~Shu.
\newblock A {Brualdi--Hoffman--Tur\'{a}n} problem on cycles.
\newblock {\em European J. Combin.}, 120:No. 103966, 2024.

\bibitem{LiLiuZhang2025EdgeSpectralESS}
Y.~Li, H.~Liu, and S.~Zhang.
\newblock An edge-spectral {Erd\H{o}s--Stone--Simonovits} theorem and its
  stability, 2025.
\newblock arXiv:2508.15271.

\bibitem{LiLiuZhang2025ColorCritical}
Y.~Li, H.~Liu, and S.~Zhang.
\newblock Edge-spectral {Tur{\'a}n} theorems for color-critical graphs with
  applications, 2025.
\newblock arXiv:2511.15431.

\bibitem{LiLiuZhang2025NosalBooksC4}
Y.~Li, H.~Liu, and S.~Zhang.
\newblock More on {Nosal}'s spectral theorem: {Books} and {$4$}-cycles.
\newblock {\em J. Combin. Theory Ser. B}, 179:219--249, 2026.

\bibitem{LL2009}
B.~Liu and M.-H. Liu.
\newblock On the spread of the spectrum of a graph.
\newblock {\em Discrete Math.}, 309:2727--2732, 2009.

\bibitem{LLLY2025}
C.~Liu, J.~Li, S.~Li, and Y.~Yu.
\newblock A {Brualdi--Hoffman--Tur\'{a}n} problem on theta graph.
\newblock {\em Adv. in Appl. Math.}, 173:Paper No.103000, 2026.

\bibitem{LiuPikhurkoStaden2020}
H.~Liu, O.~Pikhurko, and K.~Staden.
\newblock The exact minimum number of triangles in graphs of given order and
  size.
\newblock {\em Forum of Math., Pi}, 8(e8):144, 2020.

\bibitem{LiuNing2026}
L.~Liu and B.~Ning.
\newblock Local properties of the spectral radius and {Perron} vector in
  graphs.
\newblock {\em J. Combin. Theory Ser. B}, 176:241--253, 2026.

\bibitem{LLZ2025}
Z.~Lou, L.~Lu, and M.~Zhai.
\newblock A refinement on spectral {Mantel's} theorem.
\newblock {\em European J. Combin.}, 127:Paper No. 104142, 2025.

\bibitem{LovaszSauermann2019PLMS}
L.~M. Lov{\'a}sz and L.~Sauermann.
\newblock A lower bound for the $k$-multicolored sum-free problem in
  $\mathbb{Z}_m^n$.
\newblock {\em Proc. Lond. Math. Soc.}, 119(1):55--103, 2019.

\bibitem{MY2025}
J.~Ma and L.-T. Yuan.
\newblock Supersaturation beyond color-critical graphs.
\newblock {\em Combinatorica}, 45(2):Paper No. 18, 2025.

\bibitem{Mubayi2010CountingSubstructuresI}
D.~Mubayi.
\newblock Counting substructures {I}: {C}olor critical graphs.
\newblock {\em Adv. Math.}, 225(5):2731--2740, 2010.

\bibitem{Niki2002}
V.~Nikiforov.
\newblock Some inequalities for the largest eigenvalue of a graph.
\newblock {\em Combin. Probab. Comput.}, 11(2):179--189, 2002.

\bibitem{Nikiforov2007C4FreeMaxSpectralRadius}
V.~Nikiforov.
\newblock The maximum spectral radius of {$C_4$}-free graphs of given order and
  size.
\newblock {\em Linear Algebra Appl.}, 430(11):2898--2905, 2009.

\bibitem{Niki2021}
V.~Nikiforov.
\newblock On a theorem of {Nosal}, 2021.
\newblock arXiv:2104.12171.

\bibitem{NZ2021}
B.~Ning and M.~Zhai.
\newblock Counting substructures and eigenvalues {I}: {Triangles}.
\newblock {\em European J. Combin.}, 110:Paper No. 103685, 12, 2023.

\bibitem{NZ2021b}
B.~Ning and M.~Zhai.
\newblock Counting substructures and eigenvalues {II}: {Quadrilaterals}.
\newblock {\em Electron. J. Combin.}, 32(4):Paper No. 4.1, 2025.

\bibitem{PikhurkoYilma2017Supersaturation}
O.~Pikhurko and Z.~B. Yilma.
\newblock Supersaturation problem for color-critical graphs.
\newblock {\em J. Combina. Theory, Ser. B}, 123:148--185, 2017.

\bibitem{Reiher2016}
C.~Reiher.
\newblock The clique density theorem.
\newblock {\em Ann. of Math.}, 184(3):683--707, 2016.

\bibitem{Sidorenko1993CorrelationInequality}
A.~Sidorenko.
\newblock A correlation inequality for bipartite graphs.
\newblock {\em Graphs and Combinatorics}, 9(2--4):201--204, 1993.

\bibitem{Szegedy2015SparseGraphLimits}
B.~Szegedy.
\newblock Sparse graph limits, entropy maximization and transitive graphs.
\newblock {\em arXiv preprint arXiv:1504.00858}, 2015.

\bibitem{Tao2008TensorPower}
T.~Tao.
\newblock The tensor power trick.
\newblock What's New (personal blog), 2008.
\newblock
  \url{https://terrytao.wordpress.com/2008/08/25/tricks-wiki-article-the-tensor-product-trick/}.

\bibitem{ZhaiLiLou2026}
M.~Zhai, R.~Li, and Z.~Lou.
\newblock Advances on two spectral conjectures regarding booksize of graphs,
  2026.
\newblock arXiv:2601.10163.

\bibitem{ZLS2021}
M.~Zhai, H.~Lin, and J.~Shu.
\newblock Spectral extrema of graphs with fixed size: {Cycles} and complete
  bipartite graphs.
\newblock {\em European J. Combin.}, 95:No. 103322, 2021.

\end{thebibliography}

\end{document}